\documentclass[11pt,a4paper]{article}
\usepackage[T2A,T1]{fontenc}
\usepackage[utf8]{inputenc}
\usepackage[russian,english]{babel}
\usepackage{lmodern,amsmath,amssymb,amsthm}
\usepackage[margin=22mm]{geometry}
\usepackage{microtype,booktabs,longtable,array,xcolor}
\usepackage{tikz}
\usepackage[hidelinks]{hyperref}
\newtheorem{theorem}{Theorem}[section]
\newtheorem{lemma}[theorem]{Lemma}
\newtheorem{claim}[theorem]{Claim}
\newtheorem{corollary}[theorem]{Corollary}
\theoremstyle{remark}

\newcommand{\F}{\mathbb F}
\newcommand{\R}{\mathbb R}
\newcommand{\Z}{\mathbb Z}
\newcommand{\ip}[2]{\langle #1,#2\rangle}
\newcommand{\one}{\mathbf{1}}
\DeclareMathOperator{\rk}{rk}
\newcommand{\Fp}{\mathbb F_p}
\newcommand{\Ftwo}{\mathbb F_2}
\newcommand{\pos}[1]{(#1)_+}

\newcommand{\vecx}{\mathbf{x}}
\newcommand{\vecy}{\mathbf{y}}

\title{A polynomial method over formal power series}
\author{%
Artur Akhiiarov\thanks{Laboratory of Advanced Combinatorics and Network Applications, Moscow Institute of Physics and Technology (State University), Moscow, Russia.}
\and
Mikhail Koshelev\thanks{Moscow Institute of Physics and Technology, Moscow State University.}
\and
Andrey Kupavskii\thanks{Moscow Institute of Physics and Technology, Russia, St. Petersburg State University.}
\and
Andrei Raigorodskii\thanks{Department of Discrete Mathematics and Laboratory of Advanced Combinatorics and Network Applications, Moscow Institute of Physics and Technology (State University), Moscow, Russia; Department of Mathematical Statistics and Random Processes, Faculty of Mechanics and Mathematics, Lomonosov Moscow State University, Moscow, Russia; Caucasus Mathematical Center, Adyghe State University, Maykop, Republic of Adygea, Russia.}%
}
\date{}
\begin{document}
\maketitle
\begin{abstract}
We give new upper bounds on the sizes of families of integer vectors with prescribed coordinate multiplicities and a forbidden scalar product. Our method uses coefficient functions over formal power series and recovers the Frankl--Wilson bound for binary vectors. For vectors in $\{-1,0,1\}^n$, we obtain an exponential improvement over the bounds obtained using the ordinary polynomial method. For certain parameters our bounds are sharp up to a polynomial factor. We use the new bounds to make progress on several questions in combinatorial geometry. In particular, we improve the lower bound for Borsuk's partition problem. We construct finite subsets of $\R^n$ requiring at least $(1.37+o(1))^{\sqrt n}$ parts of smaller diameter for all sufficiently large $n$, and obtain a base of $1.3839$ along a subsequence of dimensions.
\end{abstract}

\section{Introduction}
Write $[n]=\{1,\ldots,n\}$. For $1\le k\le n$ and $L\subseteq\{0,\ldots,k-1\}$, put 
$$m(n,k,L) = \max \Big\{|\mathcal F|: \mathcal F\subseteq\binom{[n]}k, \ |A\cap B|\in L \text{ for distinct } A,B\in\mathcal F\Big\}.$$
Understanding the behaviour of $m(n,k,L)$ in general is a hopelessly complicated task. As $L$ varies, this notion includes questions from extremal set theory, design theory and coding theory. Building on some earlier work, Deza, Erd\H{o}s and Frankl \cite{DezaErdosFrankl1978} suggested the general question of determining $m(n,k,L)$ and proved some general upper bounds on it. An earlier paper of Ray--Chaudhuri and Wilson \cite{RayChaudhuriWilson1975}, written in the language of designs, also provides a general upper bound for $m(n,k,L)$. These two papers introduced two now-classical methods to address this problem. The $\Delta$-system method of Deza, Erd\H os and Frankl gives fine-grained bounds and structural results depending on $L$, albeit only for $n>n_0(k)$. The linear-algebraic method of 
Ray--Chaudhuri and Wilson
gives a universal bound $$m(n,k,L)\le\binom n{|L|}.$$ It works for all $n,k,L$, although the bound is generally not optimal. Frankl and Wilson \cite{FranklWilson1981} proved a modular counterpart: if all intersections of distinct members lie in $s\le k$ residue classes modulo a prime $p$, none congruent to $k$, then $|\mathcal F|\le\binom ns$. For fixed $k$ and sufficiently large $n$, (a stronger version of) the $\Delta$-system method can also recover the polynomial exponents of such modular bounds, even without a prime number restriction; see F\"uredi's survey \cite[Section~9.6]{Furedi1991}.

A particularly well-studied case is the Erd\H{o}s--S\'os problem of forbidding one intersection size $t$. For convenience, in what follows we write
\[
 m(n,k,t):=m\bigl(n,k,\{0,\ldots,k-1\}\setminus\{t\}\bigr),
 \qquad 0\le t<k.
\]
For the Erd\H os--S\'os problem, combinatorial arguments give tight results in certain ranges. In an influential paper that further developed the $\Delta$-system method, Frankl and F\"uredi \cite{FranklFuredi1985} proved that, for fixed $k\ge2t+2$ and sufficiently large $n$,
\[
 m(n,k,t)=\binom{n-t-1}{k-t-1}.
\]

A largest family consists of all $k$-sets containing a fixed $(t+1)$-set. For fixed $t<k\le2t+1$, partial $t$-designs can give larger families: take a design of blocks of size $2k-t-1$ whose pairwise intersections have size less than $t$. Then take all $k$-subsets in each block. This construction, in particular, has size $\Theta(n^t)$ and matches the general Frankl--F\"uredi bound $O_k(n^{\max\{t,k-t-1\}})$ for $k<2t+1$. The optimal leading constant remains open in general, although under certain arithmetic restrictions Frankl and F\"uredi proved that this construction attains the maximum.

Ellis, Keller and Lifshitz \cite{EllisKellerLifshitz2024} used the junta method to resolve the problem in a broad range for constant $t$. They showed that the largest $t$-avoiding families are actually $(t+1)$-intersecting families whenever $t$ and $\eta>0$ are fixed, $n$ is sufficiently large, and $\eta n<k<(1/2-\eta)n$. This, in particular, covers a part of the non-trivial range for the $(t+1)$-intersection problem, that is, the regime when  the extremal examples need no longer consist of all $k$-sets containing a fixed $(t+1)$-set. Spread approximations of Kupavskii and Zakharov \cite{KupavskiiZakharov2024} gave the first exact results for $n,k,t$ that are polynomially related, and, importantly, extended the results to general domains that are somewhat quasirandom. The Erd\H os--S\'os problem could be seen as a question on the independence number of Johnson graphs. For a recent survey of related questions on Johnson graphs, see Raigorodskii and Sinelnikov-Murylev~\cite{RaigorodskiiSinelnikovMurylev2025}.

The aforementioned results provide very strong extremal and structural information in the respective regimes of the parameters, but they do not cover the regime in which both $k$ and $t$ are linear in $n$. This regime, however, was historically the most important for geometric applications. Here, we have the Frankl--Wilson theorem \cite{FranklWilson1981}, which gives $m(n,k,t)\le\binom n{k-t-1}$ when $k-t$ is a prime power and $k\ge2t+1$. Bobu, Kupriyanov and Raigorodskii \cite[Section~2]{BobuKupriyanovRaigorodskii2016} showed that this bound is tight up to a polynomial factor for $2t\le k\le n/2.$
The prime power constraint is substantial and not easy to get rid of. Frankl and R\"odl \cite{FranklRodl1987} proved a weaker bound without this arithmetic assumption: for fixed $0<\rho<1$ and $\eta>0$, there is $\varepsilon>0$ such that
\[
 m(n,k,t)\le(1-\varepsilon)^n\binom nk
\]
whenever $n$ is sufficiently large, $k=\lfloor\rho n\rfloor$, and $(\max\{0,2\rho-1\}+\eta)n<t<(\rho-\eta)n$. This result had several important applications in Euclidean Ramsey theory. The bounds that could be extracted from their argument are quantitatively very weak. Keevash and Long \cite{KeevashLong2017} subsequently derived the conclusion from the Frankl--Wilson theorem using dependent random choice (and also with quantitatively weak bounds). Kupavskii, Sagdeev and Zakharov \cite[Theorem~7]{KupavskiiSagdeevZakharov2026} subsequently used the same idea of reducing Frankl--R\"odl to Frankl--Wilson, but with a much more efficient reduction: their concatenation argument loses a factor of only three or four in the exponent, depending on whether $k-t$ is odd or even.

In the linear-algebraic method, the size of the family is bounded by the dimension of a certain linear space. The classical linear-algebraic argument of Frankl and Wilson uses inclusion matrices. Alon, Babai and Suzuki \cite{AlonBabaiSuzuki1991} developed a formulation using multilinear polynomials. Although different, the two methods work with the same space of functions: on indicator vectors $\prod_{i\in S}y_i=\mathbf{1}_{S\subseteq E}$. For a systematic treatment of both languages, see Babai and Frankl \cite[Chapters~5 and~7]{BabaiFrankl2022}.

The geometric questions that motivate our work are the chromatic number of Euclidean space and Borsuk's partition problem. The former asks how many colours are needed to colour $\R^n$ without a monochromatic pair at distance $1$. We now know that the chromatic number grows exponentially with the dimension. The latter asks for the least number $f(n)$ of parts of smaller diameter into which every bounded subset of $\R^n$ of positive diameter can be partitioned. Borsuk's original question was whether $f(n)\le n+1$ \cite{Borsuk1933}. Kahn and Kalai disproved this and established $f(n)\ge c^{\sqrt n}$ for some $c>1$ and all sufficiently large $n$ \cite{KahnKalai1993}. Until recently, the lower bounds for both problems came from Frankl--Wilson-type arguments applied to forbidden-intersection problems on integer vectors.

Raigorodskii improved the exponents in both problems using configurations in $\{-1,0,1\}^n$ \cite{Raigorodskii1999,Raigorodskii2000}. With prescribed nonnegative multiplicities $k_{-1}+k_0+k_1=n$, the vectors have common squared norm $u=k_{-1}+k_1$, and $\|\mathbf{x}-\mathbf{y}\|^2=2u-2\ip{\mathbf{x}}{\mathbf{y}}$. Thus forbidding a scalar product forbids a distance. If $V$ is the corresponding layer and $B$ bounds the size of every independent set in its forbidden-scalar-product graph, then $|V|/B$ is a lower bound for the chromatic number. To bound the independence number of these graphs, Raigorodskii used a variant of the polynomial method.  Raigorodskii \cite[Theorem~11]{Raigorodskii2007} also extended it to arbitrary finite integer alphabets and a prescribed scalar product.

In this paper, we introduce another language for the linear algebraic method, using coefficients of formal power series over $\F_p$, for prime $p$. For $\{0,1\}$, it recovers the Frankl--Wilson bound. For larger alphabets, it allows us to use the arithmetic structure of the alphabet and obtain exponentially better bounds than the ordinary polynomial method in a wide range of parameters. We show that, for certain parameters, it provides bounds sharp up to polynomial factors for $\{-1,0,1\}$-families. The construction is a signed analogue of the classical Ahlswede--Khachatrian families~\cite{AhlswedeKhachatrian1997}. In the set version, one takes all $k$-sets meeting a fixed $(s+2r)$-set in at least $s+r$ points. Any two of them then have at least $s$ common points in that set. For signed vectors one prescribes the number of minus signs in one block and of plus signs in its complement. This limits the number of negative contributions to a scalar product~\cite{GutermanEtAl2010,AkhiiarovBobuRaigorodskii2025}. Combining our estimates via the concatenation argument of Kupavskii, Sagdeev and Zakharov \cite[Lemma~2]{KupavskiiSagdeevZakharov2026} yields exponentially better bounds for certain parameter sets.

{\bf Part~I} develops the method and examines its sharpness. Section~\ref{sec:binary} gives polynomial and matrix proofs of the Frankl--Wilson bound. Section~\ref{sec:general} extends the method to a finite integer alphabet. We now describe the families of vectors involved. 

For a fixed finite nonempty alphabet $A\subseteq\Z$ and nonnegative integer multiplicities $\mathbf{k}=(k_a)_{a\in A}$ with $\sum_{a\in A}k_a=n\ge1$, put
\[
 V(n,\mathbf{k})=\{\mathbf{x}\in A^n:\#\{i:x_i=a\}=k_a\text{ for every }a\in A\}.
\]
For an integer $t$, define
\[
 m(n,\mathbf{k},t)=\max\left\{|\mathcal F|:
 \begin{array}{l}
 \mathcal F\subseteq V(n,\mathbf{k}),\\
 \ip{\mathbf{x}}{\mathbf{y}}\ne t\quad(\mathbf{x},\mathbf{y}\in\mathcal F,\ \mathbf{x}\ne\mathbf{y})
 \end{array}\right\}.
\]
Note that every vector in $V(n,\mathbf{k})$ has the same squared norm $u=\sum_{a\in A}a^2k_a$. Our main result provides bounds on $m(n,\mathbf{k},t)$ that depend on the composition of the vectors and on the arithmetic of the alphabet. In Section~\ref{sec:combinatorial} we first give a concatenation bound that applies even when $u-t$ is composite. We then compare our bound for $\{-1,0,1\}$-vectors with the ordinary polynomial bound and show that our bound is sharp for certain parameters. Appendix~\ref{app:combinatorial} gives the concatenation example and the proof of sharpness. 

{\bf Part~II} begins with Borsuk's partition problem, where we obtain $f(n)\ge(1.37+o(1))^{\sqrt n}$ in all sufficiently large dimensions and the base $1.3839$ along a subsequence. We then treat the Nelson--Hadwiger problem, orthogonality and diameter graphs of prescribed girth, and chromatic numbers of spheres of prescribed radius. Further applications concern families with a forbidden scalar product, colourings with several prescribed forbidden distances, 
forbidden monochromatic triangles and regular simplices, and Borsuk partitions on spheres. Section~\ref{sec:optimization} explains how the numerical bounds are obtained.

\part{The method}

\section{The binary case: the Frankl--Wilson bound}\label{sec:binary}

Let us start by stating the classical Frankl--Wilson theorem \cite{FranklWilson1981}.

\begin{theorem}[Frankl--Wilson]\label{thm:binary}
Let $0\le t<k\le n$, suppose that $q=k-t$ is a prime power and $k\ge2t+1$, and let $\mathcal F\subseteq\binom{[n]}k$. If $|E\cap F|\ne t$ for all distinct $E,F\in\mathcal F$, then
\[
 |\mathcal F|\le\binom n{q-1}.
\]
\end{theorem}

The role of $k\ge2t+1$ is to make $t=k-q$ the only integer in $\{0,\ldots,k-1\}$ congruent to $k$ modulo $q$. More generally, the same bound holds whenever $q-1\le k$ and distinct members satisfy $|E\cap F|\not\equiv k\pmod q$, without the restriction $k\ge2t+1$. This is the modular form of the theorem \cite{FranklWilson1981}. For more general sets of allowed residues modulo a prime power, Babai, Frankl, Kutin and \v{S}tefankovi\v{c} \cite{BabaiFranklKutinStefankovic2001} give bounds using separating polynomials whose degrees depend on the residue pattern.

We give two formulations of the proof of Theorem~\ref{thm:binary}. In both, we identify each $k$-set with its indicator vector and write $q=p^\ell$, where $p$ is prime.

\subsection{The polynomial formulation}

The proof uses the following standard fact. For a field $K$, let $\mathcal P_{\le d}$ be the space of functions on the $k$-uniform layer represented by polynomials of total degree at most $d$. Then
\begin{equation}\label{lem:dimension}
 \dim_K\mathcal P_{\le d}\le\binom nd\qquad(0\le d\le k\le n).
\end{equation}
For $d\le\min\{k,n-k\}$, equality holds over every field \cite[Corollary~3.1]{HegedusRonyai2003}. If $n-k<d\le k$, the full function space has dimension $\binom nk\le\binom nd$. Equivalently, the inclusion matrix whose columns are indexed by all subsets of size at most $d$ has rank at most $\binom nd$ over $K$.

We work in the ring of formal power series in the indeterminate $z$,
\[
 \F_p[[z]]
 =\left\{\sum_{r\ge0}a_rz^r:a_r\in\F_p\right\}.
\]
For a series $f$, write $[z^r]f$ for its coefficient of $z^r$. Addition is coefficientwise, and multiplication is defined by
\[
 [z^r](fg)=\sum_{i=0}^r([z^i]f)([z^{r-i}]g).
\]
A series is divisible by $z^q$ if its coefficients of degrees $0,\ldots,q-1$ vanish. Two series are congruent modulo $z^q$ if their coefficients in these degrees agree.

Put $T=1+z$. It is invertible in $\F_p[[z]]$, with $T^{-1}=1-z+z^2-\cdots$, so $T^m$ is defined for every integer $m$.

For a nonzero integer $d$, let $v_p(d)$ be the largest integer $v\ge0$ such that $p^v$ divides $d$, and put $\rho_p(d)=p^{v_p(d)}$.
The following coefficient identity will also be used in Section~\ref{sec:general}.

\begin{lemma}\label{lem:selector}
Let $p$ be prime and $q$ an arbitrary positive integer. For integers $0\le m<q+\rho_p(q)$,
\begin{equation}\label{eq:extended-selector}
 [z^{q-1}]T^{m-1}
 =\begin{cases}
 (-1)^{q-1},&m=0,\\
 1,&m=q,\\
 0,&\text{otherwise}.
 \end{cases}
\end{equation}
For all $m\in\Z$, this coefficient vanishes when $\rho_p(q)\nmid m$.
\end{lemma}
\begin{proof}
Write $m-1=h\rho_p(q)+r$ with $h\in\Z$ and $0\le r<\rho_p(q)$. In characteristic $p$,
\[
 T^{m-1}=(1+z^{\rho_p(q)})^hT^r.
\]
This identity also holds for $h<0$, since every integer power of $1+z^{\rho_p(q)}$ is a formal power series in $z^{\rho_p(q)}$.
The coefficient of $z^{q-1}$ vanishes unless $r=\rho_p(q)-1$, equivalently $\rho_p(q)\mid m$. In that case it equals
\[
 [z^{q/\rho_p(q)-1}]T^{m/\rho_p(q)-1}.
\]
If $0<m<q+\rho_p(q)$ and $\rho_p(q)\mid m$, then $m/\rho_p(q)$ is a positive integer at most $q/\rho_p(q)$. The last coefficient is therefore one at $m=q$ and zero otherwise. At $m=0$, the expansion of $T^{-1}$ gives $(-1)^{q-1}$.
\end{proof}

\begin{proof}[Proof of Theorem~\ref{thm:binary}]
For each $\mathbf{x}\in\binom{[n]}k$, define a function on the same layer by
\[
 \phi_{\mathbf{x}}(\mathbf{y})=[z^{q-1}]T^{k-1-\ip{\mathbf{x}}{\mathbf{y}}}.
\]
Since $q=p^\ell$, we have $(-1)^{q-1}=1$ in $\F_p$. Thus Lemma~\ref{lem:selector} gives $\phi_{\mathbf{x}}(\mathbf{x})=1$. For distinct $\mathbf{x},\mathbf{y}\in\mathcal F$, $k\ge2t+1$ gives $0<k-\ip{\mathbf{x}}{\mathbf{y}}\le k<2q$, and the forbidden intersection gives $k-\ip{\mathbf{x}}{\mathbf{y}}\ne q$. As $\rho_p(q)=q$, the lemma gives $\phi_{\mathbf{x}}(\mathbf{y})=0$. Thus the functions $\phi_{\mathbf{x}}$ with $\mathbf{x}\in\mathcal F$ are linearly independent over $\F_p$.

To bound the dimension of their span, we expand $T^{-\ip{\mathbf{x}}{\mathbf{y}}}$ coordinatewise. For $S\subseteq[n]$, write $x_S=\prod_{i\in S}x_i$ and $y_S=\prod_{i\in S}y_i$. For $a,b\in\{0,1\}$, we have $T^{-ab}=1-(z/T)ab$. Multiplying over the coordinates gives
\begin{equation}\label{eq:binaryexpansion}
 T^{k-1-\ip{\mathbf{x}}{\mathbf{y}}}
 =\sum_{S\subseteq[n]}(-z)^{|S|}T^{k-1-|S|}x_Sy_S.
\end{equation}
Note that every factor $T^{k-1-|S|}$ is a series with no negative powers of $z$, even when its exponent is negative. Consequently, only the terms with $|S|<q$ can contribute to the coefficient of $z^{q-1}$. Extracting that coefficient expresses $\phi_{\mathbf{x}}$ as a linear combination over $\F_p$ of the monomials $y_S$ with $|S|\le q-1$. Thus $\phi_{\mathbf{x}}\in\mathcal P_{\le q-1}$, and \eqref{lem:dimension} gives $|\mathcal F|\le\dim_{\F_p}\mathcal P_{\le q-1}\le\binom n{q-1}$.
\end{proof}

\subsection{The matrix formulation}

\begin{proof}[Proof of Theorem~\ref{thm:binary} in matrix language]
Form the matrix $H$ on the $k$-uniform layer with entries
\[
 H_{\mathbf{x},\mathbf{y}}
 =[z^{q-1}]T^{k-1-\ip{\mathbf{x}}{\mathbf{y}}}
 =\phi_{\mathbf{x}}(\mathbf{y}).
\]
The values of $\phi_{\mathbf{x}}$ on $\mathcal F$ give
\begin{equation}\label{eq:identity}
 H[\mathcal F,\mathcal F]=I_{|\mathcal F|},
 \qquad |\mathcal F|\le\rk_{\F_p}H.
\end{equation}

For $a,b\in\{0,1\}$, the entries $T^{-ab}$ form the coordinate matrix
\begin{equation}\label{eq:binary-M}
 M_{\{0,1\}}=
 \begin{pmatrix}1&1\\1&T^{-1}\end{pmatrix}
 =
 \begin{pmatrix}1\\1\end{pmatrix}
 \begin{pmatrix}1&1\end{pmatrix}
 -\frac zT
 \begin{pmatrix}0\\1\end{pmatrix}
 \begin{pmatrix}0&1\end{pmatrix}.
\end{equation}
The $(\mathbf{x},\mathbf{y})$ entry of its $n$th Kronecker power is the product of the coordinate entries, namely $T^{-\ip{\mathbf{x}}{\mathbf{y}}}$. Choosing the second summand in~\eqref{eq:binary-M} at the coordinates in $S$ gives the term $(-z/T)^{|S|}x_Sy_S$. Multiplying by $T^{k-1}$ recovers~\eqref{eq:binaryexpansion}. Only the terms with $|S|\le q-1$ contribute to $[z^{q-1}]$, so the row space of $H$ lies in $\mathcal P_{\le q-1}$. Applying~\eqref{lem:dimension} gives
\[
 |\mathcal F|\le\rk_{\F_p}H\le\dim_{\F_p}\mathcal P_{\le q-1}\le\binom n{q-1}.
\]
\end{proof}

For a finite integer alphabet $A$, the same construction starts from $M_A(z)=(T^{-ab})_{a,b\in A}$.

Appendix~\ref{sec:matrix} bounds the rank of a coefficient of a tensor power and uses this bound to prove Theorem~\ref{thm:general} for integer alphabets. For consecutive integers in increasing order, the Newton weights in characteristic two are precisely the orders of the Smith invariant factors of $M_A$ (Lemma~\ref{lem:smith-consecutive}).

\section{Fixed-composition integer vectors}\label{sec:general}

Let $A$, $\mathbf{k}$ and $V(n,\mathbf{k})$ be as defined in the introduction. Fix an ordering $(a_0,\ldots,a_{s-1})$ of the letters of $A$, where $s=|A|$. All vectors in $V(n,\mathbf{k})$ have the common squared norm $u=\sum_{j=0}^{s-1}k_{a_j}a_j^2$.
Consequently,
\[
 u-\ip{\mathbf{x}}{\mathbf{y}}=\tfrac12\|\mathbf{x}-\mathbf{y}\|^2>0
 \qquad(\mathbf{x},\mathbf{y}\in V(n,\mathbf{k}),\ \mathbf{x}\ne \mathbf{y}).
\]
We continue to work in $R=\F_p[[z]]$ with $T=1+z$. We now allow $q$ to be any positive integer. In the statement of the main theorem we pay for it by using a prime power factor of $q$ rather than $q$ itself. This refinement of the Frankl--Wilson method is due to Frankl and Singhi~\cite{FrSin}. 

For a family satisfying the hypotheses of our main Theorem~\ref{thm:general}, the coefficient from Lemma~\ref{lem:selector} will again give linearly independent functions on the family. To bound the span of these functions, we need a coordinate expansion of $T^{-\ip{\vecx}{\vecy}}$. It consists of $T^{-x_iy_i}$ with $x_iy_i$ taking values in $\mathbb Z$ (rather than just $\{0,1\}$, as in the binary case), and $y_i$ being a variable. As before, we want to rewrite $T^{-x_iy_i}$ so that the variable $y_i$ is no longer in the exponent, and this forces us to work in the ring of polynomials over formal power series. As a basis for this ring, we use Newton polynomials evaluated at the points corresponding to the alphabet. This choice allows us to exploit the arithmetic of the alphabet.

Here, we present the polynomial formulation of the method. In Appendix~\ref{sec:matrix} we give the matrix formulation. One of its features is that it sheds light on the choice of basis in the polynomial ring over $R$ and connects it with the Smith normal form of the matrix  $(T^{-a\cdot b})_{a,b\in A}$.

\subsection{Preparation for the general bound}\label{subsec:prep}
We use the following valuation property, whose proof is given in Section~\ref{sec32}.
\begin{claim}\label{claimdegree}
    For every nonzero integer $d$, the value $\rho_p(d)$ is the least degree of a nonzero term in $T^d-1$.
\end{claim}

With the ordered evaluation points $T^{-a_0},\ldots,T^{-a_{s-1}}$ we associate the \emph{Newton polynomials} in $R[X]$ given by
\begin{equation}\label{eq:newton}
 P_0(X)=1,\qquad
 P_j(X)=\prod_{i<j}(X-T^{-a_i})\qquad(1\le j\le s).
\end{equation}
Each $P_j$ has degree $j$ and is \emph{monic}, meaning that its leading coefficient is $1$.
If $h<j$, then $P_j(T^{-a_h})=0$ because one of its factors is zero. If $h\ge j$, all the factors are nonzero, and each has the form
\[
 T^{-a_h}-T^{-a_i}=T^{-a_h}(1-T^{a_h-a_i}).
\]
The first factor has constant term $1$, while the second has least nonzero degree $\rho_p(a_h-a_i)$. In a product of nonzero series these degrees add, since the first nonzero coefficient is the product of the first nonzero coefficients of the factors. Hence the least nonzero degree of $P_j(T^{-a_h})$ is $\sum_{i<j}\rho_p(a_h-a_i)$.

For $0\le j<s$, the largest power of $z$ dividing every value of $P_j$ at the evaluation points is $z^{w_j}$, where
\begin{equation}\label{eq:weights}
 w_0=0,\qquad
 w_j=\min_{j\le h<s}\sum_{i<j}\rho_p(a_h-a_i)
 \qquad(1\le j<s).
\end{equation}
For $I=(i_1,\ldots,i_n)\in\{0,\ldots,s-1\}^n$, put
\[
 W(I)=\sum_{r=1}^n w_{i_r},\qquad
 H_{I}(\mathbf{y})=\prod_{r=1}^nP_{i_r}(T^{-y_r}).
\]
Every value of $H_{I}$ is divisible by $z^{W(I)}$. Thus $H_I$ can contribute to the coefficient of $z^{q-1}$ in an $R$-linear combination of these products only if $W(I)<q$.

The prescribed coordinate multiplicities of vectors in $V=V(n,\mathbf{k})$ impose a further restriction on which $H_I$ can be nonzero on our layer. Take $\mathbf{y}\in V$ and write $y_r=a_{j_r}$ for each $r\in[n]$. If $H_I(\mathbf{y})\ne0$, the zero pattern of the Newton polynomials gives $i_r\le j_r$ in every coordinate. Consequently, for every $1\le h<s$ we have
\[
 \{r:i_r\ge h\}\subseteq\{r:j_r\ge h\},\qquad
 \#\{r:j_r\ge h\}=\sum_{j=h}^{s-1}k_{a_j}.
\]
Thus, if $H_I$ is nonzero on $V$, then $I$ must belong to the collection
\begin{equation}\label{eq:admissible}
 \mathcal I=\left\{I\in\{0,\ldots,s-1\}^n:
 \#\{r:i_r\ge h\}\le\sum_{j=h}^{s-1}k_{a_j}
 \quad(1\le h<s)\right\}.
\end{equation}
We are now ready to state the main result. When $q=p^\ell$, we have $\rho_p(q)=q$ and the lower endpoint in the first hypothesis is $u-2q$.

\begin{theorem}\label{thm:general}
Let $\mathcal F\subseteq V(n,\mathbf{k})$ satisfy
\[
 \begin{aligned}
 \ip{\mathbf{x}}{\mathbf{y}}&>u-q-\rho_p(q)
 &&(\mathbf{x},\mathbf{y}\in\mathcal F),\\
 \ip{\mathbf{x}}{\mathbf{y}}&\ne u-q
 &&(\mathbf{x},\mathbf{y}\in\mathcal F,\ \mathbf{x}\ne \mathbf{y}).
 \end{aligned}
\]
With the weights and the collection $\mathcal I$ defined above,
\begin{equation}\label{eq:general-bound}
 \begin{aligned}
 |\mathcal F|&\le C_{A,\mathbf{k}}(q):=\sum_{I\in\mathcal I}(q-W(I))_+\\
 &\le C_A(n,q):=\sum_{I\in\{0,\ldots,s-1\}^n}(q-W(I))_+,
 \end{aligned}
\end{equation}
where $(t)_+=\max\{t,0\}$.
\end{theorem}
Both sums are taken for the chosen prime $p$ and ordering of $A$.
Estimating the sum defining $C_A(n,q)$ in~\eqref{eq:general-bound} gives the following corollary.
\begin{corollary}\label{cor:optimization}
Under the hypotheses of Theorem~\ref{thm:general},
\begin{equation}\label{eq:tau-bound}
 |\mathcal F|\le C_A(n,q)\le q\tau^{-q}\left(\sum_{j=0}^{s-1}\tau^{w_j}\right)^n
 \qquad(0<\tau\le1).
\end{equation}
\end{corollary}
\begin{proof}[Proof of Corollary~\ref{cor:optimization}]
For $0<\tau\le1$ and every $x\ge0$, we have $(q-x)_+\le q\tau^{x-q}$. Substituting this into~\eqref{eq:general-bound} and using $W(I)=\sum_r w_{i_r}$ gives
\[
 C_A(n,q)\le q\tau^{-q}\sum_{I}\tau^{W(I)}
 =q\tau^{-q}\left(\sum_{j=0}^{s-1}\tau^{w_j}\right)^n.
\]
\end{proof}

\subsection{Proof of the general bound}\label{sec32}
\begin{proof}[Proof of Claim~\ref{claimdegree}]
If $d>0$, write $d=p^vb$ with $p\nmid b$. The identity $T^{p^v}=1+z^{p^v}$ in characteristic $p$ gives
\[
 T^d-1=(1+z^{p^v})^b-1
 =b z^{p^v}+\text{terms of higher degree}.
\]
The displayed coefficient is nonzero in $\F_p$. If $d<0$, use
\[
 T^d-1=-T^{d}(T^{-d}-1).
\]
The series $-T^{d}$ has nonzero constant term $-1$, so multiplication by it preserves the least nonzero degree. Since $v_p(-d)=v_p(d)$, the claim follows for negative $d$ as well.
\end{proof}

\begin{proof}[Proof of Theorem~\ref{thm:general}]
We first prove the result when all letters of $A$ are nonnegative. The reduction from an arbitrary integer alphabet is given at the end. For $\mathbf{x}\in\mathcal F$, define
\[
 \phi_{\mathbf{x}}(\mathbf{y})=[z^{q-1}]T^{u-1-\ip{\mathbf{x}}{\mathbf{y}}}\qquad(\mathbf{y}\in V(n,\mathbf{k})).
\]
For distinct $\mathbf{x},\mathbf{y}\in\mathcal F$, the hypotheses give
\[
 0<u-\ip{\mathbf{x}}{\mathbf{y}}<q+\rho_p(q),
 \qquad u-\ip{\mathbf{x}}{\mathbf{y}}\ne q.
\]
By Lemma~\ref{lem:selector}, $\phi_{\mathbf{x}}(\mathbf{y})=0$, whereas $\phi_{\mathbf{x}}(\mathbf{x})=(-1)^{q-1}\ne0$ in $\F_p$. Thus the functions $\phi_{\mathbf{x}}$ are linearly independent on $\mathcal F$.
It remains to bound the dimension of their span by $\sum_{I\in\mathcal I}(q-W(I))_+$.

For each $b\in A$, divide $X^b\in\F_p[[z]][X]$ by the monic polynomial $P_s$ and expand the remainder in $P_0,\ldots,P_{s-1}$. This gives
\begin{equation}\label{eq:newton-division}
 X^b=Q_b(X)P_s(X)+\sum_{j=0}^{s-1}c_j(b)P_j(X),
 \qquad(Q_b\in \F_p[[z]][X],\ c_j(b)\in \F_p[[z]]).
\end{equation}
To perform the first division, a leading term $CX^d$ with $d\ge s$ is removed by subtracting $CX^{d-s}P_s(X)$. In the remainder, a leading term $CX^j$ is removed by subtracting $CP_j(X)$. Each operation lowers the degree and keeps all coefficients in $\F_p[[z]]$, because each divisor has leading coefficient~$1$.

Now evaluate \eqref{eq:newton-division} at $X=T^{-a}$, where $a\in A$. The factor $P_s(T^{-a})$ is zero, so
\begin{equation}\label{eq:one-coordinate}
 T^{-ab}=\sum_{j=0}^{s-1}c_j(b)P_j(T^{-a}).
\end{equation}

Apply \eqref{eq:one-coordinate} with $a=y_r$ and $b=x_r$ in each coordinate, then multiply by $T^{u-1}$. We obtain
\begin{equation}\label{eq:general-product}
 T^{u-1-\ip{\mathbf{x}}{\mathbf{y}}}=\sum_{I} b_{I}(\mathbf{x})H_{I}(\mathbf{y}),
 \qquad b_{I}(\mathbf{x})=T^{u-1}\prod_{r=1}^n c_{i_r}(x_r)\in \F_p[[z]],
\end{equation}
where the sum is over $I\in\{0,\ldots,s-1\}^n$. By the argument preceding the theorem, $H_{I}$ vanishes on $V(n,\mathbf{k})$ when $I\notin\mathcal I$. When $W(I)\ge q$, its values are divisible by $z^q$. Since $b_{I}(\mathbf{x})\in \F_p[[z]]$ has no negative powers of $z$, these terms contribute nothing to $\phi_{\mathbf{x}}$ either.

For $I\in\mathcal I$ and $e\ge0$, put $f_{I,e}(\mathbf{y})=[z^e]H_{I}(\mathbf{y})$. These functions vanish for $e<W(I)$. Extracting the coefficient of $z^{q-1}$ in \eqref{eq:general-product} gives
\begin{equation}\label{eq:convolution}
 \phi_{\mathbf{x}}(\mathbf{y})=\sum_{I\in\mathcal I}\ \sum_{e=W(I)}^{q-1}
 \bigl([z^{q-1-e}]b_{I}(\mathbf{x})\bigr)f_{I,e}(\mathbf{y}).
\end{equation}
The multiplier of each $f_{I,e}$ is a scalar in $\F_p$ depending only on $\mathbf{x}$. On $\mathcal F$, the $|\mathcal F|$ independent functions $\phi_{\mathbf{x}}$ therefore lie in the span of the restrictions of the $f_{I,e}$ in~\eqref{eq:convolution}. For each $I\in\mathcal I$, there are $(q-W(I))_+$ possible values of $e$. This proves the first inequality in \eqref{eq:general-bound} for nonnegative alphabets. Adding the nonnegative summands with $I\notin\mathcal I$ proves the second.

Finally, let $A$ be any integer alphabet, put $a_{\min}=\min A$, and replace each vector $\mathbf{x}$ by $\mathbf{x}'=\mathbf{x}-a_{\min}\mathbf{1}$. The translated alphabet has nonnegative letters. Every original vector has the same coordinate sum $s_1=\sum_jk_{a_j}a_j$, so its translated squared norm and the translated scalar products satisfy
\[
 u'=u-2a_{\min}s_1+na_{\min}^2,\qquad
 \ip{\mathbf{x}'}{\mathbf{y}'}=\ip{\mathbf{x}}{\mathbf{y}}-2a_{\min}s_1+na_{\min}^2.
\]
Hence $u'-\ip{\mathbf{x}'}{\mathbf{y}'}=u-\ip{\mathbf{x}}{\mathbf{y}}$, so both hypotheses of the theorem are preserved. The ordered multiplicities and all differences between letters are unchanged, so the weights $w_j$ and the collection $\mathcal I$ are unchanged as well. Translation is a bijection on the family, and the bound just proved for the translated alphabet gives~\eqref{eq:general-bound} for the original one.
\end{proof}

The same bound holds without a window condition if $\ip{\mathbf{x}}{\mathbf{y}}\not\equiv u\pmod{\rho_p(q)}$ for all distinct $\mathbf{x},\mathbf{y}\in\mathcal F$. Indeed, the last assertion of Lemma~\ref{lem:selector} again gives $\phi_{\mathbf{x}}(\mathbf{y})=0$ for each such pair.

\section{Combinatorial results}\label{sec:combinatorial}

\subsection{Concatenation across coordinate blocks}\label{sec:concatenation}

The concatenation technique of Kupavskii, Sagdeev and Zakharov~\cite{KupavskiiSagdeevZakharov2026} allows us to apply Theorem~\ref{thm:general} on separate blocks of coordinates. In particular, the arithmetic and scalar-product conditions can hold on the blocks even when they fail on the full layer.

Let $A$, $\mathbf{k}$ and $u$ be as in Section~\ref{sec:general}. For an integer $t<u$, put $q=u-t$ and choose a decomposition
\[
 q=q_1+\cdots+q_r,\qquad q_i\in\Z_{>0}.
\]
Choose nonnegative integer multiplicities $\mathbf{k}^{(i)}=(k_a^{(i)})_{a\in A}$ such that $\sum_{i=1}^r k_a^{(i)}=k_a$ for each $a\in A$, and put
\[
 n_i=\sum_{a\in A}k_a^{(i)},\qquad
 u_i=\sum_{a\in A}a^2k_a^{(i)}.
\]
Partition $[n]$ into blocks of sizes $n_1,\ldots,n_r$, and consider the vectors having profile $\mathbf{k}^{(i)}$ on block $i$. If two such vectors have scalar product $u_i-q_i$ in every block $i$, then their scalar product is $\sum_i(u_i-q_i)=t$. For instance, a sufficiently large $q$ can be written as a sum of three or four almost equal primes. We may then use a different characteristic on each block.

\begin{theorem}\label{thm:concatenation}
Suppose that $m(n_i,\mathbf{k}^{(i)},u_i-q_i)\le B_i$ for $1\le i\le r$. Then
\begin{equation}\label{eq:concatenation}
 \frac{m(n,\mathbf{k},t)}{|V(n,\mathbf{k})|}
 \le\min\left\{1,\sum_{i=1}^r\frac{B_i}{|V(n_i,\mathbf{k}^{(i)})|}\right\}.
\end{equation}
\end{theorem}

\begin{proof}[Proof of Theorem~\ref{thm:concatenation}]
Let $\mathcal F\subseteq V(n,\mathbf{k})$ avoid scalar product $t$. Averaging over ordered partitions $[n]=N_1\sqcup\cdots\sqcup N_r$ with $|N_i|=n_i$ gives a partition for which the subfamily $\mathcal F'$ of vectors having $k_a^{(i)}$ copies of each letter $a$ on $N_i$ satisfies
\begin{equation}\label{eq:concatenation-average}
 \frac{|\mathcal F'|}{\prod_i|V(n_i,\mathbf{k}^{(i)})|}
 \ge\frac{|\mathcal F|}{|V(n,\mathbf{k})|}.
\end{equation}
Identify $\mathcal F'$ with a subset of $\prod_i V(n_i,\mathbf{k}^{(i)})$ by restricting each vector to the blocks $N_i$. On $V(n_i,\mathbf{k}^{(i)})$, join distinct vectors of scalar product $u_i-q_i$ to form a graph $G_i$. The hypothesis gives $\alpha(G_i)\le B_i$.

Consider a path in this product whose $i$th step changes only block $i$, along an edge of $G_i$. Its endpoints $x,y$ have scalar product
\[
 \langle x,y\rangle=\sum_{i=1}^r(u_i-q_i)=t.
\]
Thus $\mathcal F'$ contains no such path. By the orthogonal-path bound of~\cite[Lemma~2]{KupavskiiSagdeevZakharov2026} and~\eqref{eq:concatenation-average},
\[
 \frac{|\mathcal F|}{|V(n,\mathbf{k})|}
 \le\frac{|\mathcal F'|}{\prod_i|V(n_i,\mathbf{k}^{(i)})|}
 \le\sum_{i=1}^r\frac{\alpha(G_i)}{|V(n_i,\mathbf{k}^{(i)})|}
 \le\sum_{i=1}^r\frac{B_i}{|V(n_i,\mathbf{k}^{(i)})|}.
\]
Taking the maximum over $\mathcal F$ proves~\eqref{eq:concatenation}.
\end{proof}

Appendix~\ref{app:concatenation-example} shows how concatenation improves the bound when $u-t=6^\ell$. 

\subsection{The ternary bound}\label{sec:ternary-bound}

Order the alphabet as $(0,1,-1)$, let $q$ be a power of two, and put $u=k_{-1}+k_1$. Suppose that distinct members of $\mathcal F$ avoid scalar product $u-q$ and that all pairwise scalar products exceed $u-2q$. The Newton polynomials from Section~\ref{sec:general} are
\[
 P_0(X)=1,\qquad P_1(X)=X-1,\qquad P_2(X)=(X-1)(X-T^{-1}).
\]
In characteristic two, $T-T^{-1}=z^2/T$. Hence $P_2(T)=z^3/T$, while $P_2(1)=P_2(T^{-1})=0$. This gives the three weights $0,1,3$.

Write the coefficient index as a signed vector by recording $0,+1,-1$ for the factors $P_0,P_1,P_2$, respectively. Let $i$ count its $+1$ coordinates and $j$ its $-1$ coordinates. A product can be nonzero on the layer only if $j\le k_{-1}$ and $i+j\le u$, since $P_2$ vanishes unless the input letter is $-1$ and both nonconstant factors vanish at the letter $0$. Dropping these layer restrictions, we compare the two sets of indices
\begin{align}
 \text{ordinary degrees}\quad & i,j\ge0,\quad i+j\le n,\quad i+2j<q,\label{eq:oldcount}\\
 \text{formal-series weights}\quad & i,j\ge0,\quad i+j\le n,\quad i+3j<q.\label{eq:ternary-unrestricted-count}
\end{align}
The ordinary polynomial bound counts each index in~\eqref{eq:oldcount} once~\cite{GutermanEtAl2010,AkhiiarovBobuRaigorodskii2025}. The bound $C_{\{-1,0,1\}}(n,q)$ from Theorem~\ref{thm:general} counts $q-i-3j$ coefficient functions for each index in~\eqref{eq:ternary-unrestricted-count}. Figure~\ref{fig:ternary-regions} shows the two regions for $q\le n$. For fixed proportions $i/n$ and $j/n$, the exponential rate of the number of coefficient indices is given by the entropy of the three proportions $i/n$, $j/n$ and $1-(i+j)/n$. Thus, for a fixed limiting value of $q/n$, the two bounds amount, on the exponential scale, to maximizing the same function over the two shaded triangles. When the maximizer on the larger triangle lies outside the darker one, strict concavity gives a smaller maximum on the darker triangle.

\begin{figure}[htbp]
\centering
\begin{tikzpicture}[x=6.6cm,y=6.6cm,font=\small]
 \fill[black!14] (0,0) -- (1,0) -- (0,0.5) -- cycle;
 \fill[blue!32!gray!55] (0,0) -- (1,0) -- (0,0.33333) -- cycle;
 \draw[thick,dashed,black!75] (0,0.5) -- (1,0);
 \draw[thick,dashed,blue!45!black] (0,0.33333) -- (1,0);
 \draw[->] (0,0) -- (1.09,0);
 \draw[->] (0,0) -- (0,0.60);
 \node[below left] at (0,0) {$0$};
 \draw (1,0) -- ++(0,-0.012) node[below] {$q$};
 \draw (0,0.5) -- ++(-0.012,0) node[left] {$q/2$};
 \draw (0,0.33333) -- ++(-0.012,0) node[left] {$q/3$};
 \node[anchor=west] at (0.02,0.56) {$j$ (number of $-1$ coordinates)};
 \node[anchor=north] at (0.57,-0.070) {$i$ (number of $+1$ coordinates)};
 \filldraw[fill=black!14,draw=black!75] (0.92,0.39) rectangle (0.98,0.42);
 \node[anchor=west] at (1.00,0.405) {$i+2j<q$};
 \filldraw[fill=blue!32!gray!55,draw=blue!45!black] (0.92,0.26) rectangle (0.98,0.29);
 \node[anchor=west] at (1.00,0.275) {$i+3j<q$};
\end{tikzpicture}
\caption{Admissible coefficient indices for $q\le n$. The ordinary count uses the whole shaded triangle, and the formal-series count uses the darker triangle.}
\label{fig:ternary-regions}
\end{figure}

\subsection{Sharpness for negative scalar products}\label{sec:sharp}

For the following proportions the signed-vector construction of~\cite{GutermanEtAl2010,AkhiiarovBobuRaigorodskii2025} attains the exponential rate of our upper bound.

\begin{theorem}\label{thm:ternarysharp}
Let $A=\{-1,0,1\}$ and put
\[
 (k'_{-1},k'_0,k'_1)=\left(\frac5{26},\frac8{13},\frac5{26}\right),
 \qquad t'=-\frac2{13}.
\]
For every sequence with $k_a/n\to k'_a$ for each $a\in A$ and $t/n\to t'$, we have
\begin{equation}\label{eq:sharp-rate}
 \liminf_{n\to\infty}m(n,\mathbf{k},t)^{1/n}\ge\frac{13}{8}\,2^{7/13}.
\end{equation}
There is a sequence of dimensions and parameters with these limiting proportions for which $k_{-1}+k_1-t$ is a power of two and $m(n,\mathbf{k},t)^{1/n}$ tends to the right-hand side of~\eqref{eq:sharp-rate}.
\end{theorem}
The details of the calculations for both the upper and lower bounds are given in Appendix~\ref{app:sharpness-proof}.

\clearpage

\part{Applications}

\section{Geometric and extremal consequences}
\label{sec:applications}
We begin with Borsuk's partition problem and the chromatic number of Euclidean space. For the first problem, two vectors with scalar product $t$ are mapped to points whose distance equals the diameter of the constructed set. For the second problem, they give points at unit distance. Thus at most $m(n,\mathbf{k},t)$ vectors of the layer $V(n,\mathbf{k})$ can belong to one part or colour class, and the ratio $|V(n,\mathbf{k})|/m(n,\mathbf{k},t)$ gives a lower bound on the number of parts or colours. We then turn to forbidden scalar products, several forbidden distances, monochromatic triangles and regular simplices, and finally to Borsuk's problem on spheres. Section~\ref{sec:optimization} explains the optimization setup behind the numerical bounds.

Some of our geometric bounds, notably on the chromatic number of the space, have been superseded by recent work. Davies~\cite[Theorem~1.4]{Davies2026} proved that $\chi(\R^n)\ge(2+o(1))^n$. By a theorem of Tomon~\cite[Theorem~1.3]{Tomon2026}, the same lower bound can be attained by unit-distance graphs of arbitrarily large girth. Still, we present our bounds for these problems, both for the method they illustrate and because the families of integer vectors behind the constructions have become objects of study in their own right. For $\{-1,0,1\}$-vectors, for instance, how large can a family with prescribed numbers of plus and minus signs be if no two distinct members have scalar product $t$? This question has been studied in detail~\cite{GutermanEtAl2010,PonomarenkoRaigorodskii2014,AkhiiarovBobuRaigorodskii2025}. Our method gives new upper bounds for such families and applies to other finite integer alphabets as well.
\subsection{The Borsuk problem}
\label{sec:short-borsuk}

Let $f(n)$ be the least number of sets of strictly smaller diameter
needed to partition every bounded set of positive diameter in $\R^n$.
For surveys of Borsuk's partition problem, see Raigorodskii~\cite{Raigorodskii2001,RaigorodskiiBorsuk2007,RaigorodskiiLectures2007}.
The vertices of a regular simplex show that $f(n)\ge n+1$.
Borsuk~\cite{Borsuk1933} asked whether equality always holds. Kahn and
Kalai~\cite{KahnKalai1993} disproved this in high dimensions, showing that
$f(n)>1.203^{\sqrt n}$ for all sufficiently large $n$.
Raigorodskii~\cite{Raigorodskii1999} subsequently obtained
$f(n)\ge(1.2255\ldots+o(1))^{\sqrt n}$.
The upper bound $f(n)\le(\sqrt{3/2}+o(1))^n$ follows from
Schramm's illumination theorem~\cite{Schramm1988}. Bourgain and Lindenstrauss~\cite{BourgainLindenstrauss1991} gave a different proof.

For a layer $V(d,\mathbf{k})$, suppose that
\begin{equation}\label{eq:borsuk-condition}
 t<\frac{\left(\sum_{a\in A}ak_a\right)^2}{d}.
\end{equation}
The construction of Kahn and Kalai~\cite{KahnKalai1993}, as used by
Raigorodskii~\cite{Raigorodskii1999,Raigorodskii2001} (see also Kalai's survey~\cite{Kalai2015Borsuk}), gives
\[
 f(n_d)\ge\frac{|V(d,\mathbf{k})|}{m(d,\mathbf{k},t)},
 \qquad\text{where }n_d=\frac{(d+1)(d+2)}2.
\]
Suppose that $f(n_d)\ge(c+o(1))^d$ for every sufficiently large $d$, where $c>1$. For $n_d\le n<n_{d+1}$ we have $f(n)\ge f(n_d)$ and $d=\sqrt{2n}+O(1)$. Thus
\[
 f(n)\ge\bigl(c^{\sqrt2}+o(1)\bigr)^{\sqrt n}.
\]
Optimizing the ratio subject to~\eqref{eq:borsuk-condition} as in Section~\ref{sec:optimization}, first for all large $d$ and then along a subsequence, gives
\[
 f(n)\ge(1.3700\ldots+o(1))^{\sqrt n},
 \qquad \limsup_{n\to\infty}f(n)^{1/\sqrt n}\ge1.3839.
\]
The first bound uses the alphabets $\{-2,-1,0,1,2\}$ and $\{-3,-2,-1,0,1,2,3\}$. The second uses  $\{-3,-2,-1,0,1,2\}$.


In Guterman et al.~\cite{GutermanEtAl2010} the authors predicted, but not proved,  the base $1.389\ldots$ based on a certain construction  of an independent sets. Later, Moskva and Raigorodskii~\cite[Theorems~1 and~3]{MoskvaRaigorodskii2011} subsequently found exponentially larger independent sets in that graph. For related estimates see~\cite{HarlamovaRaigorodskii2013}.

Both lower bounds hold for the finite version of Borsuk's problem, since our constructions are finite. Write $f_{\mathrm{fin}}(n)$ for the partition number when only finite sets are considered. Clearly $f_{\mathrm{fin}}(n)\le f(n)$, and Kalai~\cite{Kalai2026FiniteInfinite} conjectures strict inequality in sufficiently large dimensions. A recent preprint of OpenAI~\cite[Theorem~1.1]{OpenAI2026BorsukNine} establishes $f(9)>10$ using the compact set of rank-one orthogonal projectors on $\R^4$, embedded in the nine-dimensional affine space of trace-one symmetric matrices. This construction, however, is infinite and does not imply $f_{\mathrm{fin}}(9)>10.$
\subsection{The Nelson--Hadwiger problem}
\label{sec:short-nh}

The Nelson--Hadwiger problem asks for the least number $\chi(\R^n)$
of colours needed to colour $\R^n$ so that points at Euclidean distance
one receive different colours. See~\cite{Soifer2009,Gardner1960,Hadwiger1961} and the surveys~\cite{Raigorodskii2001,RaigorodskiiSurvey2013}.
By the de Bruijn--Erd\H{o}s theorem~\cite{DeBruijnErdos1951}, $\chi(\R^n)$ is attained by the unit-distance graph on a finite set of points.

The classical asymptotic bounds were
\[
 \left(\frac{1+\sqrt2}{2}+o(1)\right)^n
 \le\chi(\R^n)\le(3+o(1))^n.
\]
The lower bound, which follows from the Frankl--Wilson theorem~\cite{FranklWilson1981},
settled Erd\H os's conjecture on the exponential growth of the chromatic number.
The upper bound is due to Larman and Rogers~\cite{LarmanRogers1972}. Raigorodskii's $\{-1,0,1\}$
construction~\cite{Raigorodskii2000} improved the lower exponential base from
$1.207\ldots$ to $1.239\ldots$. Very recently, Davies
\cite[Theorems~1.4--1.5]{Davies2026} proved
\[
 (2+o(1))^n\le\chi(\R^n)
 \le\left(\frac{3\sqrt3}{2}+o(1)\right)^n.
\]
His lower bound uses the equality of the measurable and ordinary chromatic numbers in dimensions at least four~\cite[Theorem~1.1]{Davies2026}. Our refinement of the linear-algebraic method improves the estimates of Frankl--Wilson and Raigorodskii for distance graphs on integer vectors over prescribed alphabets.

For any integer $t<u$, two vectors in $V(n,\mathbf{k})$ are at distance
$\sqrt{2(u-t)}$ precisely when their scalar product is $t$, since both
have squared norm $u$. Scale the layer by $1/\sqrt{2(u-t)}$. In a proper unit-distance colouring, each colour class contains at most $m(n,\mathbf{k},t)$ points of the scaled layer, giving
\[
 \chi(\R^n)\ge\frac{|V(n,\mathbf{k})|}{m(n,\mathbf{k},t)}.
\]
The two choices of scale in Section~\ref{sec:optimization} yield bounds for all sufficiently large dimensions and along a subsequence, respectively,
\[
 \chi(\R^n)\ge(1.3092\ldots+o(1))^n,
 \qquad
 \limsup_{n\to\infty}\chi(\R^n)^{1/n}\ge1.3160\ldots.
\]
The first bound uses the alphabets $\{-3,-2,-1,0,1,2\}$ and $\{-3,-2,-1,0,1,2,3\}$. The second uses the seven-letter alphabet.

\subsection{Orthogonality and diameter graphs of prescribed girth}
\label{sec:girth}

The girth of a graph is the length of its shortest cycle, with value
infinity for a forest. Fix an integer $g\ge3$. Let $O_g(n)$ be the supremum
of the chromatic numbers of finite graphs of girth at least $g$ whose vertices
are distinct unit vectors in $\R^n$ and whose edges join perpendicular vectors.
Define $D_g(n)$ in the same way, with distinct points in $\R^n$ as vertices
and with the endpoints of every edge at the diameter of the vertex set.
In both definitions we may select a subgraph of the graph of all orthogonal
pairs or all diameter pairs. The girth condition concerns the selected
subgraph. For the corresponding distinction between faithful and
non-faithful unit-distance representations, see Alon and
Kupavskii~\cite{AlonKupavskii2014}.

For unit-distance graphs, Demekhin, Raigorodskii and
Rubanov~\cite{DemekhinRaigorodskiiRubanov2013} studied large chromatic
numbers under clique and cycle restrictions, and Kupavskii~\cite{Kupavskii2014}
considered graphs of small clique number. See Raigorodskii~\cite{RaigorodskiiSurvey2014}
for a survey.
Prosanov~\cite{Prosanov2019} studied diameter graphs of large girth.
Buci\'c and Davies~\cite{BucicDavies2025} obtained lower bounds for both quantities
with bases independent of $g$.
For unit-distance graphs, Tomon~\cite[Theorem~1.3]{Tomon2026} showed that
every chromatic number of a unit-distance graph in $\R^n$ is also attained
by one of arbitrarily large girth in $\R^n$. Together with Davies's
bound~\cite[Theorem~1.4]{Davies2026}, this gives unit-distance graphs of girth at
least $g$ with chromatic number at least $(2+o(1))^n$. His argument does
not apply to orthogonality or diameter graphs.
Under condition~\eqref{eq:borsuk-condition}, the construction of
Buci\'c and Davies gives
\[
 O_g(2n+2)\ge\frac{|V(n,\mathbf{k})|}{m(n,\mathbf{k},t)},
 \qquad
 D_g\!\left(\frac{(2n+2)(2n+3)}2\right)
 \ge\frac{|V(n,\mathbf{k})|}{m(n,\mathbf{k},t)}.
\]
The interval optimization in
Section~\ref{sec:optimization} gives a lower bound of $(1.2493\ldots+o(1))^n$
for the common ratio in every sufficiently large dimension $n$.
The two ambient dimensions are asymptotic to $2n$ and $2n^2$, respectively.
Embedding either graph in larger dimensions gives the bounds in
Table~\ref{tab:girth-bounds}.

\begin{table}[htbp]
\centering
\begin{tabular}{lcc}\toprule
Quantity & Buci\'c--Davies & Present bound\\\midrule
$O_g(n)$ & $(1.0745+o(1))^n$ & $(1.1177+o(1))^n$\\
$D_g(n)$ & $(1.1070+o(1))^{\sqrt n}$
 & $(1.1704+o(1))^{\sqrt n}$\\\bottomrule
\end{tabular}
\caption{Lower bounds for the chromatic numbers of orthogonality and
diameter graphs of prescribed girth.}
\label{tab:girth-bounds}
\end{table}

\subsection{Chromatic numbers of spheres of prescribed radius}
\label{sec:spheres}

Write $S_r^{n-1}=\{\mathbf{x}\in\R^n:\|\mathbf{x}\|_2=r\}$, and let
$\chi(S_r^{n-1})$ be the least number of colours needed to colour this
sphere so that points at distance one receive different colours.
At $r=1/2$, the forbidden pairs are precisely the antipodal pairs,
so $\chi(S_{1/2}^{n-1})=2$.

For every fixed $r>1/2$, Lov\'asz~\cite{Lovasz1983} proved
$\chi(S_r^{n-1})\ge n$, and Raigorodskii~\cite{Raigorodskii2012} established
exponential growth. Kostina and Raigorodskii~\cite{KostinaRaigorodskii2015} (see also
Kostina~\cite{Kostina2019}) improved these lower bounds for $r\ge0.67\ldots$
and showed that the bound $(1.239\ldots+o(1))^n$ of~\cite{Raigorodskii2000}
holds for all $r\ge0.79\ldots$.
Rogers's covering theorem~\cite{Rogers1963} gives
$\chi(S_r^{n-1})\le(2r+o(1))^n$. Prosanov~\cite{Prosanov2018} improved
this upper exponential base for $r>\sqrt5/2$.
Bikeev, Voronov and Raigorodskii~\cite{BikeevVoronovRaigorodskii2026}
obtained a further upper bound using the improved colouring bounds for
Euclidean space.

For $t<u$, the standard construction used by Kostina and Raigorodskii gives
\[
 \chi(S_r^{n-1})\ge\frac{|V(n,\mathbf{k})|}{m(n,\mathbf{k},t)}
 \qquad\left(
 u-\frac{\left(\sum_{a\in A}ak_a\right)^2}{n}\le2r^2(u-t)
 \right).
\]
Indeed, subtracting the coordinate mean puts the layer in the hyperplane
orthogonal to $\one$. Scaling by $1/\sqrt{2(u-t)}$ then takes pairs of scalar
product $t$ to pairs at unit distance. The displayed inequality says that
the centred, scaled layer has radius at most $r$. Translating it in the
direction of $\one$ places it on $S_r^{n-1}$ and preserves all distances.
Optimizing subject to this inequality as in Section~\ref{sec:optimization}
gives the bounds in Table~\ref{tab:sphere-bases}.

\begin{table}[htbp]
\centering
\begin{tabular}{ccc}\toprule
$r$ & Earlier bounds~\cite{Raigorodskii2000,Raigorodskii2012,KostinaRaigorodskii2015,Kostina2019} & Present bound\\\midrule
$0.60$ & $1.0485$ & $1.0796$\\
$0.65$ & $1.0896$ & $1.1534$\\
$0.70$ & $1.1448$ & $1.2368$\\
$0.75$ & $1.2105$ & $1.3022$\\
$0.80$ & $1.2395$ & $1.3092$\\\bottomrule
\end{tabular}
\caption{Lower exponential bases $c$ in
$\chi(S_r^{n-1})\ge(c+o(1))^n$ for fixed $r$ and all sufficiently
large $n$.}
\label{tab:sphere-bases}
\end{table}

At $r=0.60$, the base $1.0796$ uses $A=\{0,1,2\}$ and both $q=2^\ell$ and $q=3\cdot2^\ell$. With powers of two alone, the alphabets $\{0,1,2\}$ and $\{0,1,2,3,4\}$ and the composition-restricted count $C_{A,\mathbf{k}}(q)$ give the base $1.0761$. The bounds at the other radii use several alphabets. At $r=0.65$, five alphabets cover different ranges of $q/n$.

\subsection{Forbidden scalar products}
\label{sec:ternary}

Even when the full layer contains pairs with scalar product at most $u-2q$,
we can apply Theorem~\ref{thm:general} on a subset where all scalar products
exceed this threshold. A counting argument of Ponomarenko and
Raigorodskii~\cite{PonomarenkoRaigorodskii2014} transfers the resulting bound
to the full layer.

\begin{lemma}\label{lem:subset-count}
Let $\varnothing\ne W\subseteq V(n,\mathbf{k})$ and suppose that
$q=u-t=2^\ell$ for an integer $\ell\ge1$.
If $u-\ip{\mathbf{x}}{\mathbf{y}}<2q$ for all $\mathbf{x},\mathbf{y}\in W$, then
\[
 m(n,\mathbf{k},t)\le
 \frac{|V(n,\mathbf{k})|}{|W|}\,C_{A,\mathbf{k}}(q).
\]
\end{lemma}

\begin{proof}
Let $\mathcal F\subseteq V(n,\mathbf{k})$ avoid scalar product $t$.
For a permutation $\pi$ of $\{1,\ldots,n\}$, put
$(\pi\mathbf{x})_i=x_{\pi(i)}$. Count pairs $(\mathbf{x},\pi)$
with $\mathbf{x}\in\mathcal F$ and $\pi\mathbf{x}\in W$.

For fixed $\pi$, the images form a subfamily of $W$ avoiding scalar product $t$.
The hypothesis on $W$ and Theorem~\ref{thm:general} therefore
bound the number of such pairs by $C_{A,\mathbf{k}}(q)$.
For fixed $\mathbf{x}\in\mathcal F$ and $\mathbf{w}\in W$, exactly
$\prod_{a\in A}k_a!$ permutations satisfy $\pi\mathbf{x}=\mathbf{w}$,
obtained by matching positions carrying the same letter. Thus
\[
 |\mathcal F|\,|W|\prod_{a\in A}k_a!
 \le n!\,C_{A,\mathbf{k}}(q).
\]
Using $|V(n,\mathbf{k})|=n!/\prod_{a\in A}k_a!$
and maximizing over $\mathcal F$ proves the lemma.
\end{proof}

We construct $W$ by partitioning the coordinate positions and prescribing
the letter counts in each part so that all pairwise scalar products
exceed $u-2q$, as in~\cite{PonomarenkoRaigorodskii2014}. Thus $|W|$ is
a product of multinomial coefficients, one for each part.
For $A=\{-1,0,1\}$, we take a sequence with $q=2^\ell$, $k_a/n\to k'_a$
and $t/n\to t'$ as in Section~\ref{sec:optimization}. Optimizing the
bound in Lemma~\ref{lem:subset-count} gives Table~\ref{tab:ternary-bases}.

\begin{table}[htbp]
\centering
\begin{tabular}{cccc}\toprule
$(k'_{-1},k'_0,k'_1)$ & $t'$ & Ponomarenko--Raigorodskii & Present bound\\\midrule
$(0.05,0.75,0.20)$ & $1/8$ & $1.6443$ & $1.6248$\\
$(0.05,0.75,0.20)$ & $1/20$ & $1.7084$ & $1.6659$\\
$(0.05,0.65,0.30)$ & $3/20$ & $1.7876$ & $1.7431$\\
$(0.10,0.60,0.30)$ & $1/5$ & $1.9955$ & $1.9458$\\
$(0.10,0.60,0.30)$ & $1/15$ & $2.0760$ & $1.9763$\\
$(0.25,0.50,0.25)$ & $1/4$ & $2.3128$ & $2.2342$\\
$(0.25,0.50,0.25)$ & $1/6$ & $2.3220$ & $2.2104$\\\bottomrule
\end{tabular}
\caption{Upper bounds for the lower limit of $m(n,\mathbf{k},t)^{1/n}$
along sequences with $u-t$ a power of two, $k_a/n\to k'_a$ and
$t/n\to t'$.}
\label{tab:ternary-bases}
\end{table}

The estimates of Ponomarenko and Raigorodskii hold for admissible
parameter choices in every sufficiently large dimension. Our estimates in
the table hold along the chosen subsequences.

\subsection{Several prescribed forbidden distances}
\label{sec:multiple}

For a finite set $\mathcal L\subset(0,\infty)$, let
$\chi(\R^n;\mathcal L)$ be the least number of colours needed to
colour $\R^n$ so that points at any distance in $\mathcal L$ receive
different colours. We consider the prescribed sets
\[
 \mathcal L_r=\{1,\sqrt2,\ldots,\sqrt r\}\qquad(r=1,2,\ldots).
\]
For each fixed $r$, we study the exponential growth as $n\to\infty$.
Polynomial constructions of Raigorodskii and
Shitova~\cite{RaigorodskiiShitova2008} and Gorskaya, Mitricheva,
Protasov and Raigorodskii~\cite{GorskayaMitrichevaProtasovRaigorodskii2009}
give lower bounds for these sets.
Naslund's partition-rank method~\cite{Naslund2023}
gives further exponential lower bounds.

For a finite set $S\subset\Z$, define
\[
 m(n,\mathbf{k},S)=\max\left\{|\mathcal F|:
 \begin{array}{l}
 \mathcal F\subseteq V(n,\mathbf{k}),\\
 \ip{\mathbf{x}}{\mathbf{y}}\notin S\quad(\mathbf{x},\mathbf{y}\in\mathcal F,\ \mathbf{x}\ne\mathbf{y})
 \end{array}\right\}.
\]
Thus $m(n,\mathbf{k},\{t\})=m(n,\mathbf{k},t)$.

\begin{theorem}\label{thm:multiple-window}
Let $p$ be a prime, let $r\ge1$ be an integer, and let $q=p^\ell$,
where $\ell\ge1$ is an integer. Suppose that
\[
 \max_{\mathbf{x},\mathbf{y}\in V(n,\mathbf{k})}\bigl(u-\ip{\mathbf{x}}{\mathbf{y}}\bigr)<(r+1)q.
\]
With the notation of Section~\ref{sec:general} in characteristic $p$, we have
\[
 m\bigl(n,\mathbf{k},\{u-q,u-2q,\ldots,u-rq\}\bigr)
 \le C_{A,\mathbf{k}}(q)\le C_A(n,q).
\]
In particular,
\[
 m\bigl(n,\mathbf{k},\{u-q,u-2q,\ldots,u-rq\}\bigr)
 \le q\tau^{-q}\left(\sum_{j=0}^{s-1}\tau^{w_j}\right)^n
 \qquad(0<\tau\le1).
\]
\end{theorem}

The positive multiples of $q$ below $(r+1)q$ are precisely $q,2q,\ldots,rq$.
Forbidding the listed scalar products therefore gives the same separation
modulo $q$ as in Theorem~\ref{thm:general}. The proof is in
Appendix~\ref{app:multiple-window}.

Scaling $V(n,\mathbf{k})$ by $(2q)^{-1/2}$ makes the scalar products
$u-q,\ldots,u-rq$ correspond precisely to the distances in
$\mathcal L_r$, giving
\[
 \chi(\R^n;\mathcal L_r)
 \ge\frac{|V(n,\mathbf{k})|}
 {m\bigl(n,\mathbf{k},\{u-q,\ldots,u-rq\}\bigr)}.
\]
We apply Theorem~\ref{thm:multiple-window} with $p=2$ and $q=2^\ell$,
and carry out the interval optimization of Section~\ref{sec:optimization}
under the theorem's bound on $u-\ip{\mathbf{x}}{\mathbf{y}}$.
The larger window $(r+1)q$ makes larger alphabets useful. The bounds
in Table~\ref{tab:multiple-bases} for $r=5,10,20$ use alphabets with
up to $22$, $30$, and $38$ letters, respectively.

\begin{table}[htbp]
\centering
\begin{tabular}{ccc}\toprule
$r$ & Naslund & Present bound\\\midrule
$5$ & $2.0130$ & $2.1936$\\
$10$ & $2.6916$ & $2.9504$\\
$20$ & $3.6930$ & $4.0617$\\\bottomrule
\end{tabular}
\caption{Lower exponential bases for $\chi(\R^n;\mathcal L_r)$
with $r$ fixed and $n\to\infty$. Both bounds hold in every
sufficiently large dimension.}
\label{tab:multiple-bases}
\end{table}

\subsection{Forbidden triangles and regular simplices}
\label{sec:forbidden-simplices}

For a finite configuration $M$, let $\chi(\R^n,M)$ be the least
number of colours needed to colour $\R^n$ without a monochromatic
isometric copy of $M$. Write $\Delta_k$ for the vertices of a regular
$k$-dimensional simplex of side length one. Thus $\Delta_2$ is an
equilateral triangle. The configuration and its dimension $k$ are fixed
as the ambient dimension tends to infinity.

Frankl and R\"odl~\cite{FranklRodl1990Simplices} proved exponential
growth of $\chi(\R^n,M)$ for every fixed nondegenerate simplex $M$.
Naslund~\cite{Naslund2020} used the slice rank method to obtain an
explicit exponential lower bound for $\chi(\R^n,\Delta_2)$.
Kupavskii, Sagdeev and Zakharov~\cite[Theorems~3--4]{KupavskiiSagdeevZakharov2026}
proved lower bounds with bases $\psi_2^{1/(k+1)}$ for $\Delta_k$,
$\psi_2^{1/2}$ for every right triangle, and $\psi_2^{1/3}$ for every
acute triangle, where
\[
 \psi_2=\max_{0\le x\le1}\frac{1+x+x^3}{1+x^2+x^4}
       =1.2395667\ldots.
\]
Their product construction~\cite[Lemma~2 and Section~3.1]{KupavskiiSagdeevZakharov2026}
relates these chromatic numbers to forbidden scalar products. For $t<u$, it gives
\[
 \chi(\R^{(k+1)d},\Delta_k)
 \ge\frac{|V(d,\mathbf{k})|}{(k+1)m(d,\mathbf{k},t)}.
\]
Indeed, scale the layer by $1/\sqrt{2(u-t)}$ and join pairs at distance
one. The independence number of this graph is $m(d,\mathbf{k},t)$.
By Lemma~\ref{lem:concat_star} below, a sufficiently large colour class
in the product of $k+1$ copies of the layer contains a star with one
unit branch in each coordinate block. The branches are orthogonal, so
the leaves form a regular $k$-simplex of side length $\sqrt2$.
Scaling the product by $1/\sqrt2$ gives the required side length.
For a right triangle, use the centre and two leaves, scaling the two
blocks to the prescribed leg lengths. For an acute triangle with side
lengths $a,b,c$, use three leaves. The squared branch lengths are
$(b^2+c^2-a^2)/2$ and its cyclic permutations, all positive since the
triangle is acute.

We combine this inequality with the ratio bound from
Section~\ref{sec:short-nh}, valid in every sufficiently large dimension.
The fixed divisor $k+1$ does not affect the exponential base.
For an arbitrary ambient dimension $n$, we take
$d=\lfloor n/(k+1)\rfloor$ and embed the construction in $\R^n$.
Consequently, for every fixed $k\ge2$ and all sufficiently large $n$,
\[
 \chi(\R^n,\Delta_k)
 \ge\bigl(1.3092^{1/(k+1)}+o(1)\bigr)^n.
\]
The corresponding bases for right and acute triangles are
$1.3092^{1/2}$ and $1.3092^{1/3}$.
Table~\ref{tab:forbidden-simplices} compares the resulting bases with
those of Kupavskii, Sagdeev and Zakharov.

\begin{table}[htbp]
\centering

\begin{tabular}{lcc}\toprule
Forbidden configuration & Earlier bounds~\cite{KupavskiiSagdeevZakharov2026} & Present bound\\\midrule
Right triangle & $1.1133$ & $1.1442$\\
Acute triangle (including $\Delta_2$) & $1.0742$ & $1.0939$\\
$\Delta_3$ & $1.0551$ & $1.0696$\\
$\Delta_4$ & $1.0438$ & $1.0553$\\
$\Delta_5$ & $1.0364$ & $1.0459$\\\bottomrule
\end{tabular}
\caption{Lower exponential bases $c$ in
$\chi(\R^n,M)\ge(c+o(1))^n$ for a fixed forbidden triangle or regular
simplex $M$ and all sufficiently large ambient dimensions $n$.
Both columns use the same product construction. Values are rounded
down to four decimal places.}
\label{tab:forbidden-simplices}
\end{table}

Davies's theorem yields stronger bounds for these configurations.
To apply the same product construction to his unit-distance graphs,
we use the star case of the concatenation lemma of Kupavskii, Sagdeev
and Zakharov.

\begin{lemma}[{\cite[Lemma 2]{KupavskiiSagdeevZakharov2026}, star case}]\label{lem:concat_star}
Let $G=(V,E)$ be a finite simple graph  and let $r\ge1$ be an integer.
If $W\subseteq V^r$ satisfies
\[
 |W|>r\alpha(G)|V|^{r-1},
\]
then there exist $x^{(0)},x^{(1)},\ldots,x^{(r)}\in W$ such that, for each $1\le i\le r$,
\[
 x_j^{(i)}=x_j^{(0)}\quad(j\ne i),
 \qquad
 \{x_i^{(0)},x_i^{(i)}\}\in E.
\]
\end{lemma}

In a Euclidean representation, each difference $x^{(i)}-x^{(0)}$
is supported on the $i$th coordinate block. This gives the orthogonal
branches used above for any finite unit-distance graph.

\begin{corollary}
For every fixed $r\ge3$,
\[
 \chi(\mathbb R^n,\Delta_{r-1})
 \ge\left(2^{1/r}+o(1)\right)^n.
\]
\end{corollary}

\begin{proof}
Davies~\cite[Theorems~1.1 and~1.4]{Davies2026} showed that the
infimum, for $d\ge4$, of $\alpha(G)/|V(G)|$ over finite nonempty unit-distance
graphs induced in $\R^d$ equals the measurable independence density
$m_1(\R^d)$, and that $m_1(\R^d)\le(2+o(1))^{-d}$.
Choose a finite nonempty set $U_d\subset\R^d$ whose unit-distance
graph $G_d$ has independence ratio within $2^{-d}$ of this infimum.
Then
\[
 \eta_d:=\frac{\alpha(G_d)}{|U_d|}\le(2+o(1))^{-d}.
\]

Consider the product $(U_d/\sqrt2)^r\subset\R^{rd}$. If a colour
class contains more than $r\eta_d|U_d|^r$ points of the product,
Lemma~\ref{lem:concat_star} gives a star in that class with $r$
orthogonal branches of length $1/\sqrt2$. Its leaves are pairwise at
distance one and hence form a monochromatic copy of $\Delta_{r-1}$.
Thus a colouring avoiding $\Delta_{r-1}$ uses at least
\[
 \frac{|U_d|^r}{r\eta_d|U_d|^r}
 =\frac{1}{r\eta_d}\ge(2+o(1))^d
\]
colours. Taking $d=\lfloor n/r\rfloor$ and embedding the product
in $\R^n$ proves the claim.
\end{proof}

The two- and three-block constructions above give the bases $2^{1/2}$
and $2^{1/3}$ for prescribed right and acute triangles, respectively.

\subsection{The Borsuk problem on spheres}
\label{sec:spherical-borsuk}

For a fixed $r>1/2$, let $f_r(n)$ be the least number of parts of
diameter strictly smaller than $1$ needed to partition every set
$X\subseteq S_r^{n-1}$ of diameter $1$.
Kupavskii and Raigorodskii~\cite{KupavskiiRaigorodskii2012BorsukSpheres}
proved that $f_r(n)>n+1$ for every sufficiently large $n$ and every
fixed $r>1/2$. Their tensor construction gives bounds of the form
$(c_r+o(1))^{n^{1/(2h)}}$, where $c_r>1$ and the positive integer $h$
depends on $r$. Kostina~\cite{Kostina2019BorsukSpheres} improved the
constants by using ternary vectors and a more precise count of the
dimension of the tensor space.
Berdnikov~\cite{Berdnikov2021BorsukSpheres} subsequently studied
such constructions for general binary and ternary compositions.

All bounds in Table~\ref{tab:spherical-borsuk} use the case $h=1$ of the
tensor construction~\cite[Sections~3 and~7]{KupavskiiRaigorodskii2012BorsukSpheres}.
For a layer $V(d,\mathbf{k})$, put $q=u-t$ and let $v$ be the common
squared norm after subtracting the coordinate mean from each entry,
\[
 v=u-\frac{\left(\sum_{a\in A}ak_a\right)^2}{d}.
\]
If $t$ occurs as a scalar product between distinct vectors of the layer,
the construction gives
\begin{equation}\label{eq:spherical-borsuk-transfer}
 f_r(n_1)\ge\frac{|V(d,\mathbf{k})|}{m(d,\mathbf{k},t)},
 \qquad n_1=\frac{(d+1)(d+2)}2,
\end{equation}
provided that
\[
 q>v,\qquad \frac{2qv-v^2}{2q^2}<r^2.
\]

We use the interval optimization of Section~\ref{sec:optimization}
with $p=2$ under these constraints. Since the centred vectors have
norm $\sqrt v$, we have
$u-\ip{\mathbf{x}}{\mathbf{y}}=\|\mathbf{x}-\mathbf{y}\|^2/2\le2v$.
Thus for $q=2^\ell$ the assumption $q>v$ gives the window required
by Theorem~\ref{thm:general}. At $r=0.670$ we also use $q=3\cdot2^\ell$.
In this case the required inequality is
$u-\ip{\mathbf{x}}{\mathbf{y}}<4\cdot2^\ell$. It suffices to check
it for a pair in which the coordinates of one vector are in increasing
order and those of the other are in decreasing order, since this pair
minimizes the scalar product on the layer.

As in Section~\ref{sec:short-borsuk}, a lower bound $(c+o(1))^d$
for the ratio in~\eqref{eq:spherical-borsuk-transfer} for every large
$d$ gives $f_r(n)\ge(c^{\sqrt2}+o(1))^{\sqrt n}$ for every large $n$.
Indeed, we may take the largest $d$ for which $(d+1)(d+2)/2\le n$,
and embed the resulting configuration in $S_r^{n-1}$.

\begin{table}[htbp]
\centering

\begin{tabular}{ccc}\toprule
$r$ & Earlier bounds~\cite{Kostina2019BorsukSpheres,Berdnikov2021BorsukSpheres} & Present bound\\\midrule
$0.670$ & $1.0515$ & $1.0852$\\
$0.690$ & $1.0983$ & $1.1641$\\
$0.700$ & $1.1356$ & $1.2357$\\
$0.705$ & $1.1679$ & $1.2956$\\
$1/\sqrt2$ & $1.2255$ & $1.3700$\\\bottomrule
\end{tabular}
\caption{Lower exponential bases for $f_r(n)$ with exponent $\sqrt n$,
for fixed $r$ and all sufficiently large $n$. The earlier column gives
the bounds of Kostina and Berdnikov normalized to this exponent.
Values are rounded down to four decimal places.}
\label{tab:spherical-borsuk}
\end{table}

The first four bounds in the last column use symmetric letter
proportions on consecutive integer alphabets of three to six letters.
At $r=1/\sqrt2$, the radius condition follows from $q>v$, since
$(2qv-v^2)/(2q^2)=\tfrac12(1-(1-v/q)^2)<\tfrac12$.
We can therefore use the finite configurations of
Section~\ref{sec:short-borsuk} to obtain the last row.
All the bounds in the table hold for finite sets.

\subsection{The optimization setup}\label{sec:optimization}
In this subsection, we briefly discuss the optimization problems that we have to deal with when deriving numerical consequences of our main result. We use the notation $V(n,\mathbf{k})$ and $m(n,\mathbf{k},t)$ defined
in the introduction.

The first four applications use the ratio $|V(n,\mathbf k)|/m(n,\mathbf k,u-q)$. We first explain how to optimize its exponential base, and then distinguish bounds valid in every large dimension from bounds along a subsequence.

Fix a prime $p$ and write
\[
 q=\lambda p^\ell,\qquad
 \frac{p^\ell}{n}\longrightarrow\sigma>0,\qquad
 \frac{k_a}{n}\longrightarrow k'_a>0,\qquad
 \sum_{a\in A}k'_a=1,
\]
where $\lambda\ge1$ is a fixed integer. The weights $w_j$ are computed in characteristic $p$. By (\ref{eq:tau-bound}) and Stirling's formula, every admissible choice gives
\[
 \frac{|V(n,\mathbf k)|}{m(n,\mathbf k,u-q)}
 \ge
 \left(
 \frac{\tau^{\lambda\sigma}}
 {\left(\sum_{j=0}^{s-1}\tau^{w_j}\right)\prod_{a\in A}(k'_a)^{k'_a}}+o(1)
 \right)^n
 \qquad(0<\tau\le1).
\]
The factor $q$ in (\ref{eq:tau-bound}) is polynomial in $n$ and does not affect the base. For a fixed scale $\sigma$, put
\[
 \Phi_p(\sigma):=
 \sup_{\substack{\lambda,A,(k'_a),\,0<\tau\le1\\\mathrm{admissible}}}
 \frac{\tau^{\lambda\sigma}}
 {\left(\sum_{j=0}^{s-1}\tau^{w_j}\right)\prod_{a\in A}(k'_a)^{k'_a}}.
\]
Here admissibility includes the hypotheses of Theorem~\ref{thm:general} and the additional conditions of the application.

For a bound valid in every sufficiently large dimension, fix $\sigma_0>0$ and choose the smallest power $p^\ell\ge\sigma_0n$. Then
\[
 \sigma_0\le\frac{p^\ell}{n}<p\sigma_0.
\]
We may optimize the alphabet, $\lambda$, the proportions and $\tau$ after this power has been chosen. Since its normalized value can vary across the interval, the resulting uniform base is
\[
 c_{\mathrm{all}}(p)=
 \sup_{\sigma_0>0}\ \inf_{\sigma\in[\sigma_0,p\sigma_0]}\Phi_p(\sigma).
\]
Here we use admissible families of parameters that persist under small changes of $\sigma$ and integer rounding, keeping strict inequalities with a margin. Compactness of the interval then makes the asymptotic estimates uniform.

For a bound along a subsequence, fix $\sigma>0$ and take $n_\ell=\lfloor p^\ell/\sigma\rfloor$. Now $p^\ell/n_\ell\to\sigma$, so only one scale needs to be considered. The resulting base is
\[
 c_{\mathrm{sub}}(p)=\sup_{\sigma>0}\Phi_p(\sigma)
 \ge c_{\mathrm{all}}(p).
\]
Thus the first optimization takes the worst scale in a multiplicative interval, whereas the second may use the best single scale.

Other applications may have a different objective or additional optimization parameters. For the unrestricted count $C_A(n,q)$, we use a greedy alphabet order: start with any letter and choose each subsequent letter to attain the minimum in (\ref{eq:weights}). This minimizes (\ref{eq:tau-bound}) for every $\tau$, so no separate optimization over orderings is needed. This conclusion does not extend to the composition-restricted count $C_{A,\mathbf k}(q)$.

Our numerical bounds use $p=2$ and integer alphabets with at most $38$ letters. Some bounds use different alphabets at different values of $q/n$. We take $q=2^\ell$ and also use $q=3\cdot2^\ell$ for the sphere colouring bound at $r=0.60$ and the spherical Borsuk bound at $r=0.670$.

\paragraph*{AI usage.}
The main result was obtained using recent GPT models. The subsequent text was then restructured/rewritten by the authors, partly in dialogue with the recent GPT models. The authors take full responsibility for the content of the paper.

\appendix

\section{The matrix formulation of the method}\label{sec:matrix}

The matrix formulation gives another way to see where the improvement in the bound comes from. First, we extend the matrix proof from Section~\ref{sec:binary} to arbitrary integer alphabets, giving another proof of Theorem~\ref{thm:general}. In the last subsection, we discuss the Smith normal form of the coordinate matrix, which motivates the weights from Section~\ref{subsec:prep} and  shows that they (and the resulting gains of the method) are not specific to Newton polynomials, but rather an intrinsic property of the 1-dimensional matrix corresponding to the alphabet.

\subsection{Tensor products with weighted rank-one terms}

Our first lemma bounds the rank of a coefficient of a tensor power. It is the matrix counterpart of counting coefficient functions.

\begin{lemma}[Coefficient-rank bound]\label{lem:rank}
Let $\mathbb K$ be a field and suppose that a matrix over $\mathbb K[[z]]$ has a decomposition
\[
 M(z)=\sum_{j=0}^{s-1}z^{w_j}L_j(z)R_j(z)^{\mathsf T},
 \qquad w_j\in\mathbb Z_{\ge0},
\]
where $L_j$ and $R_j$ are column vectors of formal power series. For $\phi\in\mathbb K[[z]]$ and integers $n\ge1$ and $d\ge0$,
\begin{equation}\label{eq:rank-lemma}
 \rk_{\mathbb K}[z^d]\bigl(\phi M(z)^{\otimes n}\bigr)
 \le
 \sum_{I\in\{0,\ldots,s-1\}^n}
 \pos{d+1-W(I)},
 \qquad W(I)=\sum_{r=1}^nw_{i_r},
\end{equation}
where $I=(i_1,\ldots,i_n)$ and $(v)_+=\max\{v,0\}$. The bound also holds after restricting the rows and columns to a layer. Any tensor term that vanishes on this restriction may be omitted.
\end{lemma}
\begin{proof}
Expand the tensor product. For each index word $I$, the corresponding term has the form
\[
 z^{W(I)}\mathcal L_I(z)\mathcal R_I(z)^{\mathsf T},
\]
where we have included $\phi$ in $\mathcal L_I$. If $W(I)>d$, the term contributes nothing. Otherwise its coefficient of $z^d$ is
\[
 \sum_{h=0}^{d-W(I)}
 ([z^h]\mathcal L_I)
 ([z^{d-W(I)-h}]\mathcal R_I)^{\mathsf T}.
\]
This is a sum of $d-W(I)+1$ matrices of rank at most one. Summing over $I$ gives~\eqref{eq:rank-lemma}. Restriction preserves the displayed decompositions and cannot increase rank.
\end{proof}

With $d=q-1$ and $\phi=T^{u-1}$, each index word of weight $W(I)<q$ contributes at most $q-W(I)$ to the rank bound. Words of weight at least $q$ disappear entirely. The larger the weights in a decomposition, the fewer terms can contribute.

\subsection{A matrix proof of the general bound}

We now extend the matrix argument to an arbitrary finite integer alphabet. Let $A$, $\mathbf k$ and $V(n,\mathbf k)$ be as in the introduction, and fix an ordering
\[
 A=(a_0,\ldots,a_{s-1}),\qquad s=|A|.
\]
We work over $\Fp[[z]]$, with a prime $p$, $T=1+z$ and an integer $q\ge1$.

By Claim~\ref{claimdegree}, $\rho_p(d)$ is the least degree of a nonzero term in $T^d-1$ for every nonzero integer $d$.

Define the Newton polynomials
\[
 P_0(X)=1,\qquad
 P_j(X)=\prod_{i<j}(X-T^{-a_i})\quad(1\le j\le s).
\]
If $h<j$, then $P_j(T^{-a_h})=0$. If $h\ge j$, each factor can be written as
\[
 T^{-a_h}-T^{-a_i}=T^{-a_h}(1-T^{a_h-a_i}).
\]
Its least nonzero degree is $\rho_p(a_h-a_i)$. The least nonzero degrees add under multiplication, since the first nonzero coefficient of a product is the product of the first nonzero coefficients. Thus the largest power of $z$ dividing every entry of the $j$th Newton evaluation column is $z^{w_j}$, where
\begin{equation}\label{eq:general-weights}
 w_0=0,\qquad
 w_j=\min_{j\le h<s}\sum_{i<j}\rho_p(a_h-a_i)
 \quad(1\le j<s).
\end{equation}

For an index word $I=(i_1,\ldots,i_n)\in\{0,\ldots,s-1\}^n$, put
\[
 W(I)=\sum_{r=1}^n w_{i_r}.
\]
Recall the collection of index words allowed by the prescribed composition,
\begin{equation}\label{eq:general-indices}
 \mathcal I=
 \left\{I\in\{0,\ldots,s-1\}^n:
 \#\{r:i_r\ge h\}\le\sum_{j=h}^{s-1}k_{a_j}
 \quad(1\le h<s)\right\}.
\end{equation}

\begin{proof}[Matrix proof of Theorem~\ref{thm:general}]
We first assume that all letters of $A$ are nonnegative. Consider the matrices
\[
 M_A=(T^{-a_h a_g})_{h,g=0}^{s-1},
 \qquad
 N=(P_j(T^{-a_h}))_{h,j=0}^{s-1}.
\]
We shall factor $M_A$ through $N$, keeping all entries in $\Fp[[z]]$.

For each nonnegative integer $b$, divide $X^b$ by the monic polynomial $P_s(X)$. The remainder has degree less than $s$ and can be expanded in $P_0,\ldots,P_{s-1}$, since these polynomials are monic of consecutive degrees. Hence
\[
 X^b=Q_b(X)P_s(X)+\sum_{j=0}^{s-1}c_j(b)P_j(X),
 \qquad Q_b\in\Fp[[z]][X],\quad c_j(b)\in\Fp[[z]].
\]
Division by $P_s$ and expansion in the monic Newton basis both preserve coefficients in $\Fp[[z]]$. Evaluating at $X=T^{-a_h}$ gives
\[
 T^{-a_h b}=\sum_{j=0}^{s-1}P_j(T^{-a_h})c_j(b).
\]
Consequently, for $C=(c_j(a_g))_{j,g=0}^{s-1}$, we have
\[
 M_A=NC.
\]

Every entry of the $j$th column of $N$ is divisible by $z^{w_j}$. Put
\[
 (L_j)_h=z^{-w_j}P_j(T^{-a_h}),
 \qquad (R_j)_g=c_j(a_g).
\]
The quotients defining $L_j$ belong to $\Fp[[z]]$. Thus
\begin{equation}\label{eq:general-decomp}
 M_A=\sum_{j=0}^{s-1}z^{w_j}L_jR_j^{\mathsf T}.
\end{equation}
Moreover, $(L_j)_h=0$ whenever $h<j$. This zero pattern will give the composition restriction in~\eqref{eq:general-indices}.

Restrict the rows and columns of $T^{u-1}M_A^{\otimes n}$ to $V(n,\mathbf k)$ and denote the resulting matrix by $\mathcal N(z)$. For $I=(i_1,\ldots,i_n)$ define two vectors indexed by the layer. Their entries are
\[
 (U_I)_{\mathbf{x}}=\prod_{r=1}^n(L_{i_r})_{h_r}
 \quad\text{if }\mathbf{x}=(a_{h_1},\ldots,a_{h_n}),
\]
\[
 (V_I)_{\mathbf{y}}=T^{u-1}\prod_{r=1}^n(R_{i_r})_{g_r}
 \quad\text{if }\mathbf{y}=(a_{g_1},\ldots,a_{g_n}).
\]
Tensoring~\eqref{eq:general-decomp} gives
\begin{equation}\label{eq:general-tensor}
 \mathcal N(z)=\sum_{I\in\{0,\ldots,s-1\}^n}
 z^{W(I)}U_I(z)V_I(z)^{\mathsf T}.
\end{equation}

If $(U_I)_{\mathbf{x}}\ne0$, the zero pattern implies $i_r\le h_r$ in every coordinate. Therefore
\[
 \{r:i_r\ge h\}\subseteq\{r:h_r\ge h\},
 \qquad
 \#\{r:h_r\ge h\}=\sum_{j=h}^{s-1}k_{a_j}.
\]
Thus $U_I$ is the zero vector on the whole layer whenever $I\notin\mathcal I$, and these terms may be omitted from~\eqref{eq:general-tensor}.

Put $H=[z^{q-1}]\mathcal N(z)$. By Lemma~\ref{lem:rank}, each remaining index contributes at most $\pos{q-W(I)}$ to the rank, so
\[
 \rk_{\Fp}H\le\sum_{I\in\mathcal I}\pos{q-W(I)}.
\]
The $(\mathbf{x},\mathbf{y})$ entry of $H$ is $[z^{q-1}]T^{u-1-\ip{\mathbf{x}}{\mathbf{y}}}$. For distinct $\mathbf{x},\mathbf{y}\in\mathcal F$ we have $u-\ip{\mathbf{x}}{\mathbf{y}}=\|\mathbf{x}-\mathbf{y}\|^2/2>0$. By hypothesis this difference is less than $q+\rho_p(q)$ and differs from $q$. Lemma~\ref{lem:selector} gives $H_{\mathbf{x},\mathbf{y}}=0$, while $H_{\mathbf{x},\mathbf{x}}=(-1)^{q-1}$. Thus $H[\mathcal F,\mathcal F]=(-1)^{q-1}I_{|\mathcal F|}$ is nonsingular. Its rank is $|\mathcal F|$, which gives the first inequality in~\eqref{eq:general-bound}. The second follows by adding nonnegative summands.

Finally, let $A$ be any integer alphabet. Put $a_{\min}=\min A$ and translate every vector to $\mathbf{x}'=\mathbf{x}-a_{\min}\one$. The translated alphabet is nonnegative. All original vectors have the same coordinate sum $s_1=\sum_jk_{a_j}a_j$, so
\[
 u'=u-2a_{\min}s_1+na_{\min}^2,
 \qquad
 \ip{\mathbf{x}'}{\mathbf{y}'}=\ip{\mathbf{x}}{\mathbf{y}}-2a_{\min}s_1+na_{\min}^2.
\]
It follows that $u'-\ip{\mathbf{x}'}{\mathbf{y}'}=u-\ip{\mathbf{x}}{\mathbf{y}}$. In particular, the coefficient matrix $H$ and its submatrix on $\mathcal F\times\mathcal F$ are unchanged under this relabelling. All differences $a_h-a_i$ and the ordered multiplicities are also unchanged, so the weights and the collection $\mathcal I$ remain the same. The nonnegative case therefore proves the result for the original alphabet.
\end{proof}

\subsection{The weights and the Smith normal form}

For consecutive integer alphabets ordered increasingly and in characteristic two, the weights $w_j$ are the orders of the invariant factors of $M_A$ over the formal-power-series ring. We first explain how this normal form enters the rank argument, and then compute it.

A square matrix $M$ over $\Fp[[z]]$ with nonzero determinant has a Smith normal form
\begin{equation}\label{eq:smith-definition}
 P M Q=D_\sigma:=\operatorname{diag}
 (z^{\sigma_0},\ldots,z^{\sigma_{s-1}}),
 \qquad 0\le\sigma_0\le\cdots\le\sigma_{s-1},
\end{equation}
where $P,Q$ are invertible over $\Fp[[z]]$. 

The integers $\sigma_j$ are determined by $M$. For example, if $\delta_r(M)$ is the minimum $z$-order of a nonzero $r\times r$ minor and $\delta_0(M)=0$, then
\[
 \delta_r(M)=\sigma_0+\cdots+\sigma_{r-1},
 \qquad \sigma_{r-1}=\delta_r(M)-\delta_{r-1}(M).
\]
Thus the invariant factor orders describe the successive divisibility of all minors, as well as that of the determinant.

To use~\eqref{eq:smith-definition}, write $M=P^{-1}D_\sigma Q^{-1}$. If $L_j$ is column $j$ of $P^{-1}$ and $R_j^{\mathsf T}$ is row $j$ of $Q^{-1}$, then
\[
 M=\sum_{j=0}^{s-1}z^{\sigma_j}L_jR_j^{\mathsf T}.
\]
This is precisely the decomposition required by Lemma~\ref{lem:rank}. The orders add under tensor products. For $I=(i_1,\ldots,i_n)$, put $\Sigma(I)=\sum_r\sigma_{i_r}$. Then
\[
 M^{\otimes n}=(P^{-1})^{\otimes n}
 \operatorname{diag}\bigl(z^{\Sigma(I)}\bigr)_I
 (Q^{-1})^{\otimes n}.
\]
Ordering the diagonal entries by their exponents gives a Smith decomposition of $M^{\otimes n}$. Lemma~\ref{lem:rank} now gives
\begin{equation}\label{eq:smith-rank-bound}
 \rk_{\Fp}[z^{q-1}]\bigl(T^{u-1}M^{\otimes n}\bigr)
 \le\sum_{I\in\{0,\ldots,s-1\}^n}\pos{q-\Sigma(I)}.
\end{equation}

We now compute the Smith normal form in characteristic two.

\begin{lemma}\label{lem:smith-consecutive}
Let $A$ consist of $s$ consecutive integers, and let $M_A=(T^{-ab})_{a,b\in A}$ over $\Ftwo[[z]]$. Then
\begin{equation}\label{eq:smith-consecutive}
 \operatorname{SNF}_{\Ftwo[[z]]}(M_A)
 =\operatorname{diag}(z^{\omega_0},\ldots,z^{\omega_{s-1}}),
 \qquad
 \omega_0=0,\quad \omega_j=\sum_{r=1}^j2^{v_2(r)}.
\end{equation}
\end{lemma}
\begin{proof}
First order the alphabet as $a_{\min},a_{\min}+1,\ldots,a_{\min}+s-1$. Row and column permutations do not change the Smith normal form. Set
\[
 B=(T^{-hb})_{h,b=0}^{s-1},
 \qquad \Delta_{a_{\min}}=\operatorname{diag}(T^{-a_{\min}h})_{h=0}^{s-1}.
\]
Since
\[
 T^{-(a_{\min}+h)(a_{\min}+b)}=T^{-a_{\min}^2}T^{-a_{\min}h}T^{-a_{\min}b}T^{-hb},
 \qquad M_A=T^{-a_{\min}^2}\Delta_{a_{\min}} B\Delta_{a_{\min}},
\]
the matrices $M_A$ and $B$ are equivalent over $\Ftwo[[z]]$, since all the additional diagonal factors are units. It suffices to compute the Smith form of $B$.

The matrix $B$ is a Vandermonde matrix with evaluation points $1,T^{-1},\ldots,T^{-(s-1)}$. Define
\[
 \widetilde P_0(X)=1,\qquad
 \widetilde P_j(X)=\prod_{i=0}^{j-1}(X-T^{-i}).
\]
Let $K_{ij}=[X^i]\widetilde P_j(X)$ for $0\le i,j<s$. Since these polynomials are monic of consecutive degrees, $K$ is upper triangular with ones on the diagonal and is invertible over $\Ftwo[[z]]$. The change from monomials to Newton polynomials gives
\[
 BK=N,\qquad N_{h,j}=\widetilde P_j(T^{-h}).
\]
The matrix $N$ is lower triangular. For $h\ge j$,
\begin{equation}\label{eq:smith-newton-entry}
 N_{h,j}=T^{-hj}\prod_{r=h-j+1}^{h}(1-T^r).
\end{equation}

For a nonzero formal power series $f$, let $\nu_z(f)$ denote the least degree of a nonzero term of $f$. Claim~\ref{claimdegree} gives $\nu_z(T^d-1)=\rho_2(d)=2^{v_2(d)}$ for every nonzero integer $d$.

In a product of nonzero series the first nonzero coefficients multiply, so the orders add. Equation~\eqref{eq:smith-newton-entry} consequently gives
\[
 \nu_z(N_{h,j})=\sum_{r=h-j+1}^{h}2^{v_2(r)}
 \quad(h\ge j),\qquad
 \nu_z(N_{j,j})=\omega_j.
\]
We show that every entry in column $j$ is divisible by $z^{\omega_j}$. This requires the interval inequality
\begin{equation}\label{eq:interval-min}
 \sum_{r=b+1}^{b+j}2^{v_2(r)}\ge\omega_j
 \qquad(b\ge0).
\end{equation}
For every positive integer $r$,
\[
 2^{v_2(r)}=1+\sum_{\ell\ge1}2^{\ell-1}\one_{\{2^\ell\mid r\}}.
\]
Indeed, if $v_2(r)=v$, the right side is $1+\sum_{\ell=1}^v2^{\ell-1}=2^v$. An interval of $j$ consecutive integers contains at least $\lfloor j/2^\ell\rfloor$ multiples of $2^\ell$. Therefore
\begin{equation*}
 \sum_{r=b+1}^{b+j}2^{v_2(r)}
 =j+\sum_{\ell\ge1}2^{\ell-1}
 \left(\left\lfloor\frac{b+j}{2^\ell}\right\rfloor
       -\left\lfloor\frac b{2^\ell}\right\rfloor\right)
 \ge j+\sum_{\ell\ge1}2^{\ell-1}
                  \left\lfloor\frac j{2^\ell}\right\rfloor
 =\omega_j.
\end{equation*}
All these sums are finite. This proves~\eqref{eq:interval-min}, with equality on the diagonal $h=j$.

Put $D=\operatorname{diag}(z^{\omega_0},\ldots,z^{\omega_{s-1}})$ and divide column $j$ of $N$ by $z^{\omega_j}$ to obtain a matrix $L$. Every entry of $L$ belongs to $\Ftwo[[z]]$. It is lower triangular, and each diagonal entry is a unit, since $\nu_z(N_{j,j})=\omega_j$. Thus $L$ is invertible and
\begin{equation}\label{eq:explicit-smith}
 BK=LD,\qquad L^{-1}BK=D.
\end{equation}
Finally, $\omega_j-\omega_{j-1}=2^{v_2(j)}>0$, so each diagonal entry of $D$ divides the next one. Equation~\eqref{eq:explicit-smith} is therefore a Smith normal form, proving~\eqref{eq:smith-consecutive}.
\end{proof}

The $j$th invariant factor is therefore equal, up to a unit, to
\[
 \prod_{r=1}^j(1-T^r).
\]
The corresponding Vandermonde product also follows from Miller and Stanton~\cite[Theorem~7(a)]{MS} on setting their $q=T^{-1}$ and $a=T^{-1}-1$. For the increasing order on $A$, the minimum defining $w_j$ in~\eqref{eq:general-weights} is attained at $h=j$ by~\eqref{eq:interval-min}. Thus $w_j=\omega_j$, the order of the $j$th invariant factor of $M_A$.

\section{Two ternary examples}\label{app:combinatorial}

\subsection{Concatenation for a composite difference}\label{app:concatenation-example}

Let $A=\{-1,0,1\}$ and, for $\ell\ge3$, put
\[
 q=6^\ell,\qquad n=\frac{3q}{2},\qquad
 (k_{-1},k_0,k_1)=\left(\frac n4,\frac n2,\frac n4\right),
 \qquad t=-\frac q4.
\]
Then $u=n/2$, $u-t=q$, and
\[
 |V(n,\mathbf{k})|
 =\frac{n!}{(n/4)!^2(n/2)!}
 =(2\sqrt2+o(1))^n.
\]
The minimum scalar product on the full layer is $-n/2$. Since $\rho_p(q)\le3^\ell=o(q)$ for every prime $p$, this minimum is below $t-\rho_p(q)$ for large $\ell$. We combine one block in characteristic two with four blocks in odd characteristics to obtain the exponential improvement.

Let $P$ be the largest power of two not exceeding $q/2$, and write $\alpha=P/q\in(1/4,1/2]$. For all sufficiently large $\ell$, the decomposition into almost equal primes used in \cite[Section~3.3]{KupavskiiSagdeevZakharov2026} gives
\[
 q-P=q_1+q_2+q_3+q_4,\qquad
 q_i\text{ odd primes},\qquad q_i\sim(q-P)/4.
\]
Put $q_0=P$, $n_0=5P/4$ and $L=n-n_0$. Choose $n_i\in4\mathbb Z$ such that
\[
 n_i=\frac{Lq_i}{q-P}+O(1)\quad(1\le i\le4),
 \qquad \sum_{i=1}^4n_i=L.
\]
In every block take
$k_{-1}^{(i)}=k_1^{(i)}=n_i/4$, $k_0^{(i)}=n_i/2$, and $t_i=n_i/2-q_i$.
The profiles sum to $\mathbf{k}$ and $\sum_{i=0}^4t_i=t$.

The minimum scalar product in the $i$-th layer is $-n_i/2$. We have $n_0<2P$ and
\[
 \frac{n_i}{q_i}
 =\frac{3/2-5\alpha/4}{1-\alpha}+o(1)
 \in[19/12,7/4]+o(1)<2\qquad(1\le i\le4),
\]
so Theorem~\ref{thm:general} applies in each block, in characteristic two for $i=0$ and characteristic $q_i$ otherwise. With the alphabet ordered as $(0,1,-1)$, the weights are $0,1,3$ in the first block and $0,1,2$ in the other four. Corollary~\ref{cor:optimization}, with $\tau=2/3$ and $\tau=1/2$, respectively, gives the bounds
\[
 B_0=P\left(\frac{53}{27}\right)^{n_0}\left(\frac32\right)^P,
 \qquad
 B_i=q_i\left(\frac74\right)^{n_i}2^{q_i}\quad(1\le i\le4).
\]

Write $M_i=|V(n_i,\mathbf{k}^{(i)})|$ and put
\[
 H=\frac32\log2,\qquad c=H-\log\frac74,\qquad
 \delta=H-\log\frac{53}{27}-\frac45\log\frac32.
\]
Stirling's formula gives
\[
 \begin{aligned}
 \frac1n\log\frac{M_0}{B_0}
 &\ge\frac{5\alpha}{6}\delta+o(1),\\
 \frac1n\log\frac{M_i}{B_i}
 &\ge\frac{(3/2-5\alpha/4)c-(1-\alpha)\log2}{6}+o(1)
 \qquad(1\le i\le4).
 \end{aligned}
\]
Both main terms increase with $\alpha$. At $\alpha=1/4$, the first is $0.00851950\ldots$, while the second is
\[
 s:=\frac16\left(\frac{19}{16}c-\frac34\log2\right)
 =0.00837738\ldots.
\]
Theorem~\ref{thm:concatenation} therefore yields
\begin{equation}\label{eq:symmetric-composite-bound}
 m(n,\mathbf{k},t)
 \le |V(n,\mathbf{k})|e^{-sn+o(n)}
 =\left(2^{179/192}7^{19/96}+o(1)\right)^n
 =(2.8048312887\ldots+o(1))^n.
\end{equation}

Let us compare this with the ordinary polynomial method, which uses degrees $0,1,2$ as in \cite{GutermanEtAl2010,AkhiiarovBobuRaigorodskii2025}. On the same five blocks it gives no improvement over the trivial bound. Indeed, in block $0$ the term with $i=j=n_0/4$ has ordinary degree $3n_0/4=15P/16<P$, and its multinomial coefficient is exactly $M_0$. Its weight in characteristic two is $i+3j=n_0>P$, so it drops out of our count.

Alternatively, we can use the ordinary polynomial method by splitting $q$ into four almost equal odd primes and taking four almost equal blocks, each with the same symmetric profile. In each block the ratio of the prime to the block size tends to $2/3$. Minimizing $(1+\tau+\tau^2)/\tau^{2/3}$ and applying Theorem~\ref{thm:concatenation} yields the base
\[
 \gamma_{\mathrm{old}}
 =\left((2\sqrt2)^3
   \frac{1+\tau+\tau^2}{\tau^{2/3}}\right)^{1/4}
 =2.8099155559\ldots,
 \qquad \tau=\frac{\sqrt{33}-1}{8},
\]
which still exceeds $2.8048312887\ldots$ in~\eqref{eq:symmetric-composite-bound}.

\subsection{Proof of ternary sharpness}\label{app:sharpness-proof}

\begin{proof}[Proof of Theorem~\ref{thm:ternarysharp}]
\textbf{Lower bound.}
For $26\mid n$, first take $k_a=k'_a n$ and $t=-2n/13-1$. Use the signed-vector construction of~\cite{GutermanEtAl2010,AkhiiarovBobuRaigorodskii2025} with two equal blocks. Take
\[
 W=\left\{\mathbf{x}\in V(n,\mathbf{k}):
 \#\{i\le n/2:x_i=-1\}
 =\#\{i>n/2:x_i=1\}
 =\frac n{26}\right\}.
\]
In the first half, each negative coordinate product uses a minus sign from one of the two vectors. In the second half, count plus signs instead. Thus for $\mathbf{x},\mathbf{y}\in W$,
\[
 \ip{\mathbf{x}}{\mathbf{y}}\ge-\frac n{13}-\frac n{13}=t+1.
\]
The remaining counts are determined by $\mathbf{k}$, and the two halves give the same multinomial coefficient. Hence
\[
 |W|=\binom{n/2}{n/26,\,4n/13,\,2n/13}^{\!2}
     =\left(\frac{13}{8}\,2^{7/13}+o(1)\right)^n
\]
by Stirling's formula.

For arbitrary $k_a/n\to k'_a$ and $t/n\to t'$, take blocks of sizes $\lfloor n/2\rfloor$ and $\lceil n/2\rceil$. Prescribe $b=\lfloor(-t-1)/4\rfloor$ minus signs in the first block and $b$ plus signs in the second. The same count gives $\ip{\mathbf{x}}{\mathbf{y}}\ge-4b>t$. In the first block the counts of minus signs, zeros and plus signs are $b$, $\lfloor n/2\rfloor-k_1$ and $k_1-b$. In the second block the counts of plus signs, zeros and minus signs are $b$, $\lceil n/2\rceil-k_{-1}$ and $k_{-1}-b$. In both blocks their proportions tend to $(1,8,4)/13$, so the product of the two multinomial coefficients has the same exponential rate.

\medskip
\textbf{Upper bound.}
Put $q=2^\ell$ and take $n$ to be the nearest integer to $13q/7$. Choose $k_a=k'_a n+O(1)$ with $\sum_a k_a=n$, put $u=k_{-1}+k_1$ and take $t=u-q$. Then $t/n\to-2/13$. Every scalar product in the layer is at least $-u$, which exceeds $u-2q$ for all sufficiently large $\ell$, since $u/n\to5/13<7/13$. Corollary~\ref{cor:optimization} with $\tau=1/2$ gives
\[
 m(n,\mathbf{k},t)\le q\,2^q\left(\frac{13}{8}\right)^n.
\]
Here $q=7n/13+O(1)$, and the counts in the lower-bound construction differ by $O(1)$ from their limiting proportions times the block sizes. By Stirling's formula its size is of order $n^{-2}((13/8)2^{7/13})^n$, so on this sequence the two bounds differ by at most a polynomial factor.
\end{proof}

\section{Several forbidden distances}\label{app:distance-extensions}\label{app:multiple-window}

\begin{proof}[Proof of Theorem~\ref{thm:multiple-window}]
The change from one forbidden distance to several is entirely in the
modular separation step. This is the same principle as in the
Frankl--Wilson argument~\cite{FranklWilson1981}. For several distances,
see Raigorodskii and Shitova~\cite[Section~6.2, Step~5,
Proposition~6.1]{RaigorodskiiShitova2008}.

Let $\mathcal F\subseteq V(n,\mathbf{k})$ avoid the scalar products
$u-q,\ldots,u-rq$, and put
$\delta(\mathbf{x},\mathbf{y})=u-\ip{\mathbf{x}}{\mathbf{y}}$.
For distinct members of $\mathcal F$,
\[
 0<\delta(\mathbf{x},\mathbf{y})<(r+1)q,
 \qquad \delta(\mathbf{x},\mathbf{y})\notin\{q,2q,\ldots,rq\}.
\]
The excluded values are all the positive multiples of $q$ in this interval.
Since $q=p^\ell$, we have $\rho_p(q)=q$, so Lemma~\ref{lem:selector}
gives
\[
 \phi_{\mathbf{x}}(\mathbf{y})
 =[z^{q-1}]T^{\delta(\mathbf{x},\mathbf{y})-1}
 =\begin{cases}1,&\mathbf{x}=\mathbf{y},\\0,&\mathbf{x}\ne\mathbf{y},\end{cases}
 \qquad(\mathbf{x},\mathbf{y}\in\mathcal F).
\]
Thus the functions $\phi_{\mathbf{x}}$, $\mathbf{x}\in\mathcal F$, are
linearly independent. For a nonnegative alphabet, the Newton expansion
in~\eqref{eq:general-product}--\eqref{eq:convolution} expresses each of
them as a linear combination of the functions $[z^e]H_I$ with
$I\in\mathcal I$ and $W(I)\le e<q$. This gives
$|\mathcal F|\le C_{A,\mathbf{k}}(q)\le C_A(n,q)$.
For an arbitrary integer alphabet, the translation in
Section~\ref{sec32} preserves $\delta$, the weights and the multiplicities
in alphabet order. The same bounds therefore hold. Estimating $C_A(n,q)$
as in Corollary~\ref{cor:optimization} gives the last assertion.
\end{proof}

\end{document}